\documentclass[11pt]{scrartcl}
\usepackage{algorithm}
\usepackage{dsfont}
\usepackage{amssymb,amsthm}
\usepackage{mathtools}
\usepackage{enumitem}
\usepackage{lmodern}
\usepackage{todonotes}
\usepackage{tikz}
\usepackage[nomessages]{fp}
\usepackage{hyperref}
\usepackage{cleveref}

\newtheorem{theorem}{Theorem}
\newtheorem{remark}[theorem]{Remark}
\newtheorem{proposition}[theorem]{Proposition}
\newtheorem{definition}[theorem]{Definition}
\newtheorem{corollary}[theorem]{Corollary}
\newtheorem{lemma}[theorem]{Lemma}

\newcommand{\regparam}{{\frac{1}{\delta}}}
\newcommand{\domain}{\Omega}

\newcommand{\image}{I}
\newcommand{\bigO}{\mathcal{O}}

\newcommand{\N}{\mathbb{N}}

\newcommand{\R}{\mathbb{R}}

\newcommand{\Id}{\mathrm{Id}}
\newcommand{\onesymbol}{\mathds{1}}

\newcommand{\energyDensity}{\mathrm{W}}
\newcommand{\velocity}{v}
\newcommand{\velocityConstant}{w}
\newcommand{\velocityPullback}{v}

\DeclareMathOperator{\tr}{tr} 

\newcommand{\sol}{\mathcal{C}}
\newcommand{\dx}{\,\mathrm{d}}
\newcommand{\dist}{\mathrm{d}}
\renewcommand{\div}{\operatorname{div}}
\newcommand{\Hadamard}{\mathcal{H}}
\newcommand{\manifold}{\mathcal{M}}
\newcommand{\metric}{g}
\newcommand{\setHoelder}{A_{\alpha,L,w}}

\def\XXint#1#2#3{{\setbox0=\hbox{$#1{#2#3}{\int}$}
\vcenter{\hbox{$#2#3$}}\kern-.5\wd0}}

\title{Convergence of the Time Discrete Metamorphosis Model on Hadamard Manifolds}

\author{
Alexander Effland\thanks{Institute of Computer Graphics and Vision, Graz University of Technology (alexander.effland@icg.tugraz.at)}
\and Sebastian Neumayer\thanks{TU Kaiserslautern (neumayer@mathematik.uni-kl.de)}
\and Martin Rumpf\thanks{Institute for Numerical Simulation, University of Bonn (martin.rumpf@ins.uni-bonn.de)}}

\begin{document}

\maketitle

\begin{abstract}
	Continuous image morphing  is a classical task in image processing. 
	The metamorphosis model proposed by Trouv\'e, Younes and coworkers \cite{MiYo01,TrYo05} casts this problem in the frame of Riemannian geometry
	and geodesic paths between images. The associated metric in the space of images
	incorporates dissipation caused by a viscous flow transporting image intensities and its variations along motion paths.
	In many applications, images are maps from the image domain into a manifold (e.g.~in diffusion tensor imaging (DTI) the manifold of symmetric positive definite matrices with a suitable Riemannian metric).
	In this paper, we propose a generalized metamorphosis model for manifold-valued images, where the range space is a finite-dimensional Hadamard manifold.
	A corresponding time discrete version  was presented in \cite{NPS17} based on the general variational time discretization proposed in \cite{BeEf14}.
	Here, we prove the Mosco--convergence of the time discrete metamorphosis functional to the proposed manifold-valued metamorphosis model,
	which implies the convergence of time discrete geodesic paths to a geodesic path in the (time continuous) metamorphosis model.
	In particular, the existence of geodesic paths is established. In fact, images as maps into Hadamard manifold are not only 
	relevant in applications, but it is also shown that the joint convexity of the distance function -- which characterizes Hadamard manifolds --
	is a crucial ingredient to establish existence of the metamorphosis model.
\end{abstract}

\section{Introduction}\label{sec:introduction}
Image morphing amounts to computing a visually appealing transition of two images
such that image features in the reference image are mapped to corresponding image features in the
target image whenever possible.

A particular model for image morphing known as image metamorphosis
was proposed by Miller, Trouv\'e, and Younes~\cite{MiYo01,TrYo05,TrYo05a}. It is based
on the flow of diffeomorphism model and the large deformation diffeomorphic metric mapping (LDDMM),
which dates back to the work of Arnold, Dupuis, Grenander and coworkers
\cite{Ar66a,ArKh98,DuGrMi98,BeMiTr05,JoMi00,MiTrYo02,VS09,VRRC12}.
From the perspective of the flow of diffeomorphism model,
each point of the reference image is transported to the target image in
an energetically optimal way such that the image intensity is preserved
along the trajectories of the pixels.
The metamorphosis model additionally allows for image intensity modulations along the trajectories
by incorporating the magnitude of these modulations, which is reflected
by the integrated squared material derivative of the image trajectories
as a penalization term in the energy functional.
Recently, the metamorphosis model has been extended to images in reproducing kernel Hilbert spaces~\cite{RY16},
to functional shapes~\cite{CCT16}, and to discrete measures~\cite{RY13}.
For a more detailed exposition of these models we refer the reader to \cite{Younes2010,MTY15} and the references therein.

A variational time discretization of the metamorphosis model for square-integrable
images $L^2(\domain,\R^m)$ was proposed in \cite{BeEf14}. Furthermore, 
existence of discrete geodesic paths and the Mosco--convergence of the time discrete to the time continuous metamorphosis model was proven.
The time discrete metamorphosis model has successfully been applied to a variety of imaging applications
like image extrapolation~\cite{EfRu18}, B\'{e}zier interpolation~\cite{EfRu15}, color transfer~\cite{PePi17} or image interpolation in a medical context~\cite{BeBu17}.

Throughout the past years, manifold-valued images have received increased attention (see e.g.~\cite{BaBeStWe16,CiHiSc18,LeStKoCr13,DeStWe14,BeLa17}).
Some prominent applications are linked to Hadamard manifold-valued images:
\begin{itemize}[label=--]
	\item
	Diffusion tensor magnetic resonance imaging is an image acquisition method
	that incorporates in vivo magnetic resonance images of biological tissues driven by local molecular diffusion.
	The range space of the resulting images is frequently the space of symmetric and positive definite matrices \cite{BaMa94,ChTs04,FeJo07,VaBe13}.
	\item
	Retina data is commonly modeled as images with values in the manifold of univariate non-degenerate Gaussian probability distributions endowed with the Fisher metric~\cite{JeVe14,BePe16}.
	This space is isometric to a hyperbolic space, which can be exploited numerically.
\end{itemize}
This motivates a generalization of the metamorphosis model as a Riemannian model for spaces of images.
In \cite{NPS17}, the time discrete metamorphosis model was extended to the set of image~$L^2(\domain,\Hadamard)$,
where $\Hadamard$ denotes a finite-dimensional Hadamard manifold.
Recall that Hadamard manifolds are Hadamard spaces with a special Riemannian structure having non-positive sectional curvature (for details see below).
In \cite{Bac14}, it is revealed that many concepts of Banach spaces can be generalized to Hadamard spaces,
which are therefore a proper choice for the analytical treatment of algorithms for manifold-valued images.
In particular, the distance in Hadamard spaces is jointly convex, which implies weak lower semi-continuity of certain functionals involving the distance function.
Moreover, several analytic properties of Hadamard manifolds presented in~\cref{sec:preliminaries}, which are crucial for the Mosco--convergence, cease to be valid for general manifolds.

In this paper, we prove the Mosco--convergence of the manifold-valued time discrete metamorphosis energy functional originally proposed in~\cite{NPS17}
to a novel (time continuous) metamorphosis energy functional on Hadamard manifolds.
Moreover, we establish the convergence of the manifold-valued time discrete geodesic paths to geodesic paths in the
proposed manifold-valued metamorphosis model, which coincides with the original metamorphosis energy functional in the Euclidean space. 
The proof of the Mosco--convergence in \cite{BeEf14} incorporates as an essential ingredient a representation formula for images via integration of the weak material derivative 
along motion paths for the time continuous metamorphosis model in the Euclidean setting.
Here, we no longer make use of such a representation formula.
Indeed, our Mosco--convergence result can thus be considered as a stronger result even in the case of images as pointwise maps into a Euclidean space.

\medskip

\paragraph{Outline} 
The manuscript is organized as follows.
We start with a collection of required notation and symbols in the next paragraph, including the definition of Mosco--convergence.
In \cref{sec:preliminaries}, we discuss the concept of Hadamard spaces and manifolds with an emphasis on important properties of the distance map.
Furthermore, we review the classical flow of diffeomorphism and the metamorphosis model.
Here, we already prove some continuity results on the Lagrange maps associated with a motion field.
Finally, we pick up the time discrete metamorphosis model presented in~\cite{NPS17}.
\Cref{sec:manifoldMetamorphosis} is devoted to the presentation of the manifold-valued metamorphosis model.
Here, the key point is the suitable definition of a material derivative quantity, which is finally obtained using a variational inequality.
We show that the new model for manifold-valued image maps coincides with the previous model in the Euclidean case.
\Cref{sec:Extension} introduces a method to extend time-discrete image paths to time-continuous paths as the natural prerequisite to prove the convergence of the energy functionals on discrete paths to a limit energy functional on continuous paths.  
Then, in \cref{sec:Mosco}, the main result of this paper on Mosco--convergence is stated and proved.
In detail, we show the required liminf-inequality in \Cref{thm:liminf} and the existence of recovery sequences in \Cref{thm:limsup}. 
This finally implies the convergence of discrete geodesic paths in \Cref{thm:convergence}
and the existence of a geodesic path for the time continuous metamorphosis model.
The proofs generally follow the guideline from \cite{BeEf14} for the classical metamorphosis model 
with conceptual and technical modifications in order to deal with the setup of manifold-valued images.

\medskip

\paragraph{Notation} 
Throughout this paper, we assume that the image domain~$\domain\subset\R^n$ is bounded with Lipschitz boundary.
Henceforth, we denote time continuous operators by calligraphic letters
and time discrete operators by normal letters.
We denote the space of continuous functions and $k$-times continuously differentiable functions on the image domain $\domain$ by~$C^0(\overline\domain)$ and $C^k(\overline\domain)$, respectively.
H\"older spaces of order~$k$ with exponent~$\alpha$ are denoted by~$C^{k,\alpha}(\overline\domain)$.

Furthermore, we use standard notation for Lebesgue and Sobolev spaces, i.e.~$L^p(\domain)$ and $H^m(\domain)=W^{m,2}(\domain)$.
The associated norms are denoted by $\|\cdot\|_{L^p(\domain)}$ and $\|\cdot\|_{H^m(\domain)}$, respectively, and the seminorm in $H^m(\domain)$ is given by $|\cdot|_{H^m(\domain)}$.
The Sobolev (semi-)norm is defined as
\[
|f|_{H^m(\domain)}=\|D^m f\|_{L^2(\domain)}\,,\qquad
\|f\|_{H^m(\domain)}=\biggl(\sum_{j=0}^m|f|_{H^j(\domain)}^2\biggr)^\frac{1}{2}\,.
\]
The space $H_0^m(\domain)$ is the closure of $C_c^\infty(\domain)$ with respect to $\|\cdot\|_{H^m(\domain)}$.
Derivatives are always in the strong sense, if they exist, or in the weak sense otherwise.
The symmetric part of a matrix~$A\in\R^{l,l}$ is denoted by $A^\mathrm{sym}$, i.e.~$A^\mathrm{sym}=\frac{1}{2}(A+A^\top)$.
We denote by $GL^+(n)$ the elements of $GL(n)$ with positive determinant, by $\onesymbol$ the identity matrix, and by $\Id$ the identity map.

\paragraph{Mosco--convergence}
We conclude this section with a review of Mosco--convergence, which can be seen as a generalization
of $\Gamma$--convergence. For further details we refer the reader to \cite{dalMaso93,Mosco69}.
\begin{definition}[Mosco--convergence]\label{def:MoscoConv}
	Let $(X,d)$ be a metric space and let $\{J_k\}_{k \in \N}$ and $J$
	be functionals mapping from $X$ to $\overline \R$.
	Then the sequence $J_k$ is said to converge to $J$ in the sense of Mosco
	w.r.t.~the topology induced by $d$ if
	\begin{enumerate}
		\item \label{MoscoItem1}
		For every sequence $\{x_k\}_{k \in \N} \subset X$ with $x_k \rightharpoonup x\in X$ it holds that
		\begin{equation}\label{eq:Mosco1}\tag{liminf-inequality}
		J(x) \leq \liminf_{k \to \infty} J_k(x_k)\,.
		\end{equation}
		\item For every $x \in X$ there exists a recovery sequence $\{x_k\}_{k \in \N} \subset X$ such that $x_k \to x\in X$ and
		\begin{equation}\label{eq:Mosco2}\tag{limsup-inequality}
		J(x) \geq \limsup_{k \to \infty} J_k(x_k)\,.
		\end{equation}
	\end{enumerate}
	If in \ref{MoscoItem1}.~the strong convergence of $x_k$ to $x$ in the topology induced by $d$ is
	required, then $J_k$ is said to $\Gamma$-converge to $J$ w.r.t.~the topology induced by $d$.
\end{definition}

This paper is organized as follows:
In \cref{sec:preliminaries}, we briefly recall some preliminaries of Hadamard manifolds
as well as the metamorphosis model in the Euclidean case and its time discretization
on Hadamard manifolds.
Then, in \cref{sec:manifoldMetamorphosis} the novel manifold-valued metamorphosis model
is introduced and the equivalence to the original metamorphosis model in the case
of Euclidean spaces is proven.
\Cref{sec:Extension} is devoted to the temporal extension of all relevant quantities
as required for the convergence proof.
Finally, \cref{sec:Mosco} contains the precise statement of Mosco--convergence in the manifold-valued case.

\section{Review and preliminaries}\label{sec:preliminaries}
In this section, we briefly present some preliminaries of Hadamard manifolds, 
a short introduction to the metamorphosis model in the Euclidean setting~\cite{BeEf14},
and the manifold-valued time discrete metamorphosis model~\cite{NPS17}.

\subsection{Hadamard manifolds}\label{sub:Hadamard}
In what follows, a short introduction of Hadamard manifolds is provided and the space of
H\"older continuous functions on Hadamard manifolds is analyzed.
For further details we refer the reader to the books \cite{Bac14,BH1999,Jost97}.

\paragraph{Hadamard manifolds}

\begin{figure}[htb]
	\begin{tikzpicture}
	
	\coordinate[label=below left:$\bar p$] (ep) at (0,0);
	\coordinate[label=above left:$\bar x$] (ex) at (2,2);
	\coordinate[label=above left:$\bar r$] (er) at (4,4);
	\coordinate[label=below:$\bar y$] (ey) at (3,1);
	\coordinate[label=below:$\bar q$] (eq) at (6,2);
	\draw [thick] (ep) -- (er) -- (eq) -- (ep);
	\draw [thick] (ex) -- (ey);
	
	\fill (ep) circle [radius=2pt];
	\fill (eq) circle [radius=2pt];
	\fill (er) circle [radius=2pt];
	\fill (ex) circle [radius=2pt];
	\fill (ey) circle [radius=2pt];
	
	\node[anchor= west] at (1,-0.5) {Euclidean space $\R^2$};
	\node[anchor= west] at (1,-1.0) {$\bar x = \bar p + s (\bar r-\bar p),\; \bar y = \bar p + s (\bar q-\bar p)$};
	
	\coordinate[label=below left:$p$] (ep) at (8,0);
	\coordinate[label=above left:$x$] (ex) at (10,1.6);
	\coordinate[label=above left:$r$] (er) at (12,4);
	\coordinate[label=below:$y$] (ey) at (10.8,1.27);
	\coordinate[label=below:$q$] (eq) at (14,2);
	\draw [thick] (ep) to [bend right=10] (er);
	\draw [thick] (ep) to [bend left=10] (eq);
	\draw [thick] (er) to [bend right=30] (eq);
	\draw [thick] (ex) to [bend right=30] (ey);
	\node[anchor= west] at (10,-0.5) {Hadamard manifold};
	\node[anchor= west] at (10,-1.0) {$x = \gamma_{{p,r}}(s),\; y = \gamma_{{p,q}}(s)$};
	\fill (ep) circle [radius=2pt];
	\fill (eq) circle [radius=2pt];
	\fill (er) circle [radius=2pt];
	\fill (ex) circle [radius=2pt];
	\fill (ey) circle [radius=2pt];
	\end{tikzpicture}
	\label{fig:ComparisonTriangle}
	\caption{
		Comparison triangle in the Euclidean space~$\R^2$ and geodesic triangle on a Hadamard manifold, in which $d(x,y)\leq\Vert\bar x-\bar y\Vert$ is satisfied
		(Figure adapted from \cite[Figure~1.1]{Bac14}).}
\end{figure}
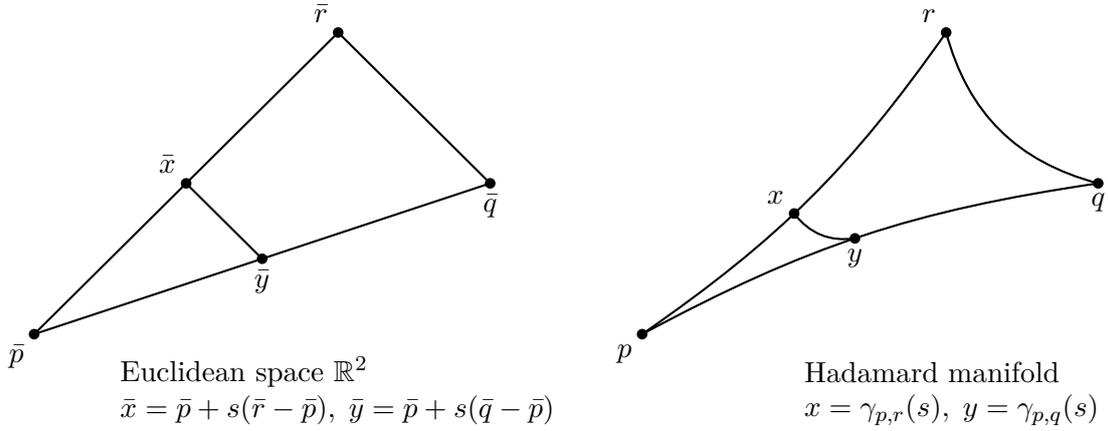

A metric space $(X,d)$ is \emph{geodesic} if every two points $x,y \in X$  
are connected by a shortest geodesic curve $\gamma_{{x,y}}\colon [0,1] \to X$, which is arclength parametrized, i.e.~for every $s,t\in[0,1]$ we have
\begin{equation}
d \bigl(\gamma_{{x,y}}(s),\gamma_{{x,y}}(t) \bigr) 
= \lvert s - t\rvert d \bigl(\gamma_{{x,y}}(0),\gamma_{{x,y}}(1) \bigr)
\label{eq:distancePropertyGeodesics}		
\end{equation}
with endpoints $\gamma_{{x,y}}(0)=x$ and $\gamma_{{x,y}}(1)=y$.
A \emph{geodesic triangle}~$\triangle(p,q,r)$ in a geodesic space~$(X,d)$
is composed of the vertices $p,q,r\in X$ and three geodesics joining these points.
The corresponding comparison triangle~$\triangle(\bar p,\bar q,\bar r)$
(which is unique up to isometries) is a triangle in the Euclidean space~$\R^2$
with vertices~$\bar p,\bar q,\bar r\in \R^2$ such that
the three line segments have the same side lengths as the corresponding geodesics of~$\triangle(p,q,r)$, i.e.
\[
d(p,q)=\Vert\bar p-\bar q\Vert\,,\quad d(p,r)=\Vert\bar p-\bar r\Vert\,,\quad d(r,q)=\Vert\bar r-\bar q\Vert\,.
\]
A complete geodesic space $(\Hadamard,d)$ is called a \emph{Hadamard space}
if for every geodesic triangle $\triangle(p,q,r)\in\Hadamard$ and 
$x\in\gamma_{p,r}$,
$y\in\gamma_{q,r}$
we have $d(x,y)\leq\Vert\bar x-\bar y\Vert$, where $\bar x$ and $\bar y$ are the corresponding points
in the comparison triangle $\triangle(\bar p, \bar q,\bar r)\in \R^2$ (see \cref{fig:ComparisonTriangle}).
Geodesic spaces satisfying the latter property are also called CAT(0) spaces.
By~\cite[Proposition~1.1.3 and Corollary~1.2.5]{Bac14} the geometric CAT(0) condition is equivalent to $(\Hadamard,d)$ being a complete geodesic space with
\begin{equation}\label{eq:reshet0}
d^2(x,v) + d^2(y,w) \le d^2(x,w) + d^2(y,v) + 2d(x,y)d(v,w)
\end{equation}
for every $x,y,v,w \in \Hadamard$.
The most prominent examples of Hadamard spaces are Hilbert spaces and \emph{Hadamard manifolds}, which are 
defined as complete simply connected Riemannian manifolds with non-positive sectional curvature.
Hyperbolic spaces and the manifold of positive definite matrices with the affine invariant metric are examples of Hadamard manifolds.
Throughout this paper, we exclusively consider finite-dimensional Hadamard manifolds, which ensure the existence of unique geodesic curves joining two arbitrary points.
Recall that the Hopf--Rinow Theorem ceases to be true for general infinite-dimensional manifolds \cite{Kl95}.

A function $f\colon \Hadamard \rightarrow \R$ is \emph{convex} if for every $x,y \in \Hadamard$ the function 
$f \circ \gamma_{{x,y}}$ is
convex, i.e.
\[
f\bigl( \gamma_{{x,y}}(t) \bigr)
\le (1-t) f\bigl( \gamma_{{x,y}} (0) \bigr)
+ t f \bigl(\gamma_{{x,y}}(1)\bigr)
\]
for all $t \in [0,1]$. 
In Hadamard spaces the distance is \emph{jointly convex} \cite[Proposition 1.1.5]{Bac14}, 
i.e.~for two geodesics $\gamma_{{x_1,x_2}},\gamma_{{y_1,y_2}}$ and $t\in[0,1]$ the relation
\begin{equation}\label{eq:ConvDist}
d\bigl(\gamma_{{x_1,x_2}}(t),\gamma_{{y_1,y_2}}(t)\bigr) \le (1-t)d(x_1,y_1)+td(x_2,y_2)
\end{equation}
holds true. Thus, geodesics are in particular uniquely determined by their endpoints.
For a bounded sequence $\{x_n\}_{n\in\N}\subset \Hadamard$, the function 
$w \colon \Hadamard \to [0, +\infty)$ defined by
\begin{equation}\label{eq:WeakConv}
w(x;\, \{x_n\}_{n\in\N}) \coloneqq \limsup_{n\to\infty} d^2(x, x_n)
\end{equation}
has a unique minimizer, which is called the \emph{asymptotic center of $\{x_n\}_{n\in\N}$}, cf.~\cite[p.~58]{Bac14}.
A sequence $\{x_n\}_{n\in\N}$ is said to \emph{converge weakly to a point $x \in \Hadamard$} if it is bounded 
and $x$ is the asymptotic center of each subsequence of $\{x_n\}_{n\in\N}$, cf.~\cite[p.~103]{Bac14}.
Then, the notion of proper and (weakly) lower semi-continuous functions is analogous to Hilbert spaces.

Next, we consider the Borel $\sigma$-algebra $\mathcal{B}$ on $\Hadamard$ on
the open and bounded set~$\Omega \subset \mathbb R^n$.
A  measurable map $f\colon\Omega\to\Hadamard$ 
belongs to $\mathcal{L}^p(\Omega,\Hadamard)$, $p \in [1,\infty]$, if 
\[
\dist_p(f,f_a) < \infty
\]
for any constant mapping $f_a(\omega)=a$ with $a\in \Hadamard$, where
$\dist_p$ is defined for two measurable maps $f$ and $g$ by
\[
\dist_p(f,g) \coloneqq 
\begin{cases}
\left(\displaystyle\int_{\Omega}d^p(f(\omega),g(\omega))\dx\omega\right)^{\frac{1}{p}}\,,\quad & p\in[1,\infty)\,,\\
\operatorname{ess\,sup}_{\omega\in\Omega}d(f(\omega),g(\omega))\,, & p  =\infty\,.
\end{cases}
\]
Using the equivalence relation $f\sim g$ if $\dist_p(f,g) = 0$,
the space $L^p(\Omega,\Hadamard)\coloneqq \mathcal{L}^p(\Omega,\Hadamard)/ \sim$ equipped with $\dist_p$ becomes a complete metric space, which
is a Hadamard space if $p=2$, cf.~\cite[Proposition 1.2.18]{Bac14}.
Finally, for $f,g$ in the \emph{weighted Bochner space} $L^2((0,1),L^2(\domain,\Hadamard),w)$ with weight
$w\in C^0([0,1] \times \domain,[c_1,c_2])$,  $0<c_1<c_2$, the metric is given by
\[
\dist_2^2(f,g)=\int_0^1\int_\domain d(f(t,x),g(t,x))^2 w(t,x)\dx x\dx t\,.
\]          
In our proposed model, we observe H\"older continuity of paths in time, which enables pointwise evaluations in time, in particular for $t=0$ and $t=1$.
Another classical property of Lebesgue spaces also transfers to the Hadamard setting.
\begin{lemma}\label{lemm:convergenceAE}
	Let $f_k \in L^2((0,1), L^2(\domain, \Hadamard),w)$ be a convergent sequence with limit $f$.
	Then there exists a subsequence which converges a.e.~in time as $k\to\infty$.
\end{lemma}
\begin{proof}
	Since the Chebyshev inequality implies the convergence in measure, we can apply \cite[Theorem 5.2.7 (i)]{Malliavin2012}.
\end{proof}
Next, we define subsets of H\"older continuous functions with fixed parameters $\alpha\in(0,1]$ and $L>0$ by
\begin{align}\label{eqSetHoelder}
\setHoelder \coloneqq \left\{f \in L^2((0,1), L^2(\domain,\Hadamard),w) \colon \dist_2(f(s),f(t)) \leq L \vert t -s \vert^{\alpha} \,\, \forall t,s \in [0,1]\right\}\,.
\end{align}
\begin{theorem}\label{thm:LipClos}
	The set $\setHoelder$ is closed and convex. In particular, $\setHoelder$ is weakly closed.
\end{theorem}
\begin{proof}
	\emph{Closedness}:
	Let $\{f_k\}_{k \in \N}\subset \setHoelder$ be a convergent sequence with limit $f$.
	By \Cref{lemm:convergenceAE} we get an a.e.~convergent subsequence denoted with the same indices.
	Assume there exists a point $t\in[0,1]$, where this sequence does not converge.
	Then, we can choose $s \in [0,1]$ arbitrarily close to~$t$ with $\dist_2(f_k(s),f(s))\to 0$ as $k\to \infty$,
	which implies
	\[
	\dist_2(f_k(t),f_l(t))\leq 2L\vert t-s\vert^{\alpha}+\dist_2(f_k(s),f_l(s))
	\]
	for all $k,l\in\N$ sufficiently large.
	Hence, the sequence converges pointwise for every $t\in[0,1]$.
	Now, the required H\"older continuity of $f$ follows from 
	\[
	d_2(f(s),f(t))=\lim_{k\to\infty} d_2(f_k(s),f_k(t))\leq L\vert t-s\vert^{\alpha}\,.
	\]
	
	\emph{Convexity}:
	Given $f_1,f_2\in\setHoelder$ we define a family of geodesic curve $r \mapsto \gamma_{{f_1(s),f_2(s)}}(r)$ for $s\in [0,1]$.
	Then, we obtain by the joint convexity of the Hadamard metric
	\begin{align*}
	\dist_2\left(\gamma_{{f_1(s),f_2(s)}}(r),\gamma_{{f_1(t),f_2(t)}}(r)\right)
	\leq (1-r)\dist_2(f_1(s),f_1(t)) + r \dist_2(f_2(s),f_2(t)) \leq L \vert t-s \vert^{\alpha}\,,
	\end{align*}
	where we used that geodesics $\gamma_{f_1(s),f_2(s)}$ can be computed pointwise for every $s\in [0,1]$.
	Finally, the weak closedness in the Bochner space follows by \cite[Lemma~3.2.1]{Bac14}.
\end{proof}
The following lemma is exploited in the proof of Mosco--convergence.
\begin{lemma}\label{lemm:stet_norm}
	Let $(\Hadamard,d)$ be a locally compact Hadamard space.
	For fixed $p\in [1,\infty)$ let $f \in L^p(\domain,\Hadamard)$ and $\{ Y_j\}_{j\in\N}\subset C^1(\overline\domain,\overline\domain)$
	be a sequence of diffeomorphisms such that $\vert \mathrm{det} (D  Y_j) \vert^{-1} \leq C$ for all $j \in \N$,
	which converges to a diffeomorphism~$ Y$ in $(L^\infty(\domain))^n$.
	Then,
	\[
	\limsup_{j \to \infty} \dist_p(f\circ  Y_j, f \circ  Y) = 0\,.
	\]
	If in addition $ Y_j$ converges to $ Y$ in $(C^{1,\alpha}(\overline\domain))^n$,
	then $\limsup_{j \to \infty}\dist_p(f\circ ( Y_j)^{-1},f\circ Y^{-1}) = 0$.
\end{lemma}
\begin{proof}
	See \cite[Corollary~3]{NPS17} and  \cite[Lemma~2.2.2]{NPS18}.
\end{proof}
The generalization of this result to the space $L^2((0,1), L^2(\domain,\Hadamard))$ is straightforward.
\begin{corollary}\label{cor:stet_norm_2}
	Let the assumptions from \cref{lemm:stet_norm} hold true and 
	let $\{f_j\}_{j \in \N}\subset L^p(\domain,\Hadamard)$, $p \in [1,\infty)$, be a sequence which converges to $f$ in $L^p(\domain,\Hadamard)$.
	Then,
	\[
	\limsup_{j\to\infty}\dist_p(f_j\circ Y_j,f\circ Y)=0 \quad\text{and}\quad\limsup_{j\to\infty}\dist_p(f_j\circ(Y_j)^{-1},f\circ Y^{-1})=0\,.
	\]
\end{corollary}
\begin{proof}
	We prove the first equation only. Using the triangle inequality, it holds that
	\[\limsup_{j\to\infty}\dist_p(f_j\circ Y_j,f\circ Y) \leq \limsup_{j\to\infty}C\dist_p(f_j,f) + \limsup_{j\to\infty}\dist_p(f\circ Y_j,f\circ Y)= 0.\]
\end{proof}
Again, the result directly generalizes to $L^2((0,1), L^2(\domain,\Hadamard))$.
For a more detailed review on Bochner spaces we refer the reader to \cite{HNVW2016}.

\subsection{Metamorphosis model in Euclidean case}
In this subsection, we briefly introduce the space of images $\image\colon\domain \to \R$ with a Riemannian structure
from the perspective of the flow of diffeomorphisms model and the metamorphosis model.
For further details we refer the reader to the literature mentioned in \cref{sec:introduction}.

\paragraph{Flow of diffeomorphisms}
In the flow of diffeomorphisms model, the temporal evolution of each pixel
of the reference image along a trajectory is determined by
a \emph{family of diffeomorphisms $(Y(t))_{t\in[0,1]}\colon\overline\domain\rightarrow\R^n$}
such that the brightness is preserved.
The \emph{brightness constancy assumption}, which is equivalent to the assertion that $t\mapsto\image(t,Y(t,x))$ is constant for a.e~$x\in\domain$,
is mathematically reflected by a vanishing material derivative~$\frac{D}{\partial t}\image=\dot\image+v\cdot D\image$
along a motion path $(\image(t))_{t\in [0,1]}$ in the space of images, where $v(t)=\dot Y(t)\circ Y^{-1}(t)$ denotes the time-dependent \emph{Eulerian velocity}.
Then, we define for a specific operator~$L$ given below the metric and the path energy associated with this family of diffeomorphisms as follows
\[
g_{Y(t)}(\dot Y(t),\dot Y(t))=\int_\domain L[v(t),v(t)]\dx x\,,\qquad
\boldsymbol{\mathcal{E}}((Y(t))_{t\in [0,1]})=\int^1_0 g_{Y(t)}(\dot Y(t),\dot Y(t))\dx t\,.
\]
Throughout this paper, we consider the \emph{higher order operator}
\begin{equation}
L[v(t),v(t)]=\tfrac{\lambda}{2}(\tr\varepsilon[v])^2+\mu\tr(\varepsilon[v]^2)+\gamma|D^m v|^2\,,
\label{eq:ellipticOperator}
\end{equation}
where $\varepsilon[v]=(Dv)^\mathrm{sym}$ refers to the symmetrized part of the Jacobian
and $m>1+\frac{n}{2}$ as well as $\lambda,\mu,\gamma>0$ are fixed constants.
This particular choice of the operator~$L$ originates from fluid mechanics,
where the metric~$g_{Y(t)}$ refers to a \emph{viscous dissipation} in a multipolar fluid model
as described in \cite{NeSi91,GrRi64,GrRi64a}.

If $Y_A$ and $Y_B$ are diffeomorphisms and the energy $\boldsymbol{\mathcal{E}}$ is finite for a general 
path $(Y(t))_{t\in[0,1]}$ with 
$Y(0)=Y_A$ and $Y(1)=Y_B$, then using the $H^m(\domain)$-coerciveness of the metric $g_{Y(t)}$ (discussed in \cite{BeKr17,DuGrMi98})
the path is already a family of diffeomorphisms. 
In addition, following~\cite{DuGrMi98} an energy minimizing velocity field~$v$ exists such that $\frac{\dx}{\dx t}Y(t,\cdot)=v(t,Y(t,\cdot))$ for 
every $t\in[0,1]$. 
Furthermore, the corresponding path~$\image$ for two input images $\image_A,\image_B\in L^2(\domain)$
has the particular form $\image(t,\cdot)=\image_A\circ Y^{-1}(t,\cdot)$.

In what follows, we investigate diffeomorphisms induced by velocity fields in the space
\[
\mathcal V \coloneqq H^m(\domain, \R^n) \cap H^1_0(\domain, \R^n)\,.
\]
The following theorem relates the norm of the induced flow to the integrated norm of the associated
velocity field.
\begin{theorem}\label{thm:DiffeoVelo}
	Let $\velocity\in L^2((0,1),\mathcal V)$ be a velocity field.
	Then, there exists a global flow $Y \in C^0([0, 1], (H^m(\domain))^n)$ such that
	\begin{equation}
	\begin{array}{rcl}
	\displaystyle\frac{\dx}{\dx t}Y(t,x)&=&\velocity(t,Y(t,x))\,,\\[0.5em]
	Y(0,x)&=&x\,,		
	\end{array}
	\label{eq:contFlow}
	\end{equation}
	for all $x \in \domain$ and a.e.~$t \in[0,1]$.
	In particular, $Y(t,\cdot)$ is a diffeomorphism for all $t \in [0, 1]$. Further, for $\alpha \in [0,m-1-\frac{n}{2})$ the following estimate holds
	\begin{align}
	\Vert Y \Vert_{C^0([0,1],C^{1,\alpha}(\overline{\domain}))} + \Vert Y^{-1} \Vert_{C^0([0,1],C^{1,\alpha}(\overline{\domain}))} 
	\leq G\left(\int_0^1 \Vert v(s,\cdot)\Vert_{C^{1,\alpha}(\overline{\domain})} \dx s\right)\label{eq:Gronwall}
	\end{align}
	for a continuous function $G(x)\coloneqq C(x+1)\exp(Cx)$.
	The solution operator from $L^2((0,1),\mathcal V)$ to $C^0([0,1], (H^m(\domain))^n)$ assigning a flow $Y$
	to every velocity field~$v$ is continuous w.r.t.~the weak topology in $L^2((0,1),\mathcal V)$ and the $C^{0}([0,1]\times \overline{\domain}))$-topology for $Y$.
\end{theorem}
\begin{proof}
	The existence follows from \cite[Theorem~4.4]{BrVi16} and the weak continuity from \cite[Theorem 9]{TrYo05a}.
	Although the first result is stated only for $\R^n$, it is still valid in our setting due to the existence of a linear and continuous extension operator from $H^m(\domain)$ to $H^m(\R^n)$,
	which is implied by Stein's extension theorem~\cite{Stein70}.
	
	The estimate for the first term in \cref{eq:Gronwall} follows from \cite[Lemma 7]{TrYo05a} and relies on Gr\"onwall's inequality.
	Let $i\in\{1,\dots,n\}$, $t\in[0,1]$ and $x,y\in\domain$.
	Taking into account \cite[Lemma 7]{TrYo05a} we obtain 
	$\Vert  Y\Vert _{C^0([0,1],C^1(\overline\domain))}\leq C\exp(C\int_0^1\Vert v(s,\cdot)\Vert _{C^{1,\alpha}(\overline\domain)}\dx s)$.
	Applying the triangle inequality and estimating the result using the H\"older continuity of $D v$ and $Y$ yields
	\begin{align*}
	&\quad\left\vert \partial_i Y(t,x)-\partial_i Y(t,y)\right\vert 
	\leq \int_0^t \bigl\vert D v(s, Y(s,x))\cdot\partial_i Y(s,x)-D v(s, Y(s,y))\cdot\partial_i Y(s,y)\bigr\vert \dx s\\
	&\leq
	\int_0^t \vert D v(s, Y(s,x))-D v(s, Y(s,y))\vert \ \vert \partial_i Y(s,x)\vert\\ &\hspace{35ex}+
	\vert Dv(t, Y(s,y))\vert \ \vert \partial_i Y(s,x)-\partial_i Y(s,y)\vert \dx s\\
	&\leq
	\int_0^t \Vert v(s,\cdot)\Vert _{C^{1,\alpha}(\overline\domain)}\Vert  Y(s,\cdot)\Vert _{C^1(\overline\domain)}^{1+\alpha}\vert x-y\vert ^\alpha+
	\Vert v(s,\cdot)\Vert _{C^1(\overline\domain)}\vert \partial_i Y(s,x)-\partial_i Y(s,y)\vert \dx s\\
	&\leq
	G\left(\int_0^1 \Vert v(s,\cdot)\Vert _{C^{1,\alpha}(\overline\domain)} \dx s\right)\, \vert x-y\vert ^\alpha + \int_0^t \Vert v(s,\cdot)\Vert _{C^1(\overline\domain)}\vert \partial_i Y(s,x)-\partial_i Y(s,y)\vert \dx s\,.
	\end{align*}
	By adapting the constant $C$ in the function~$G$, Gr\"onwall's inequality implies
	\begin{align*}
	\vert \partial_i Y(t,x)-\partial_i Y(t,y)\vert \leq 
	G\left(\int_0^1\Vert v(s,\cdot)\Vert _{C^{1,\alpha}(\overline\domain)}\dx s \right) \vert x-y\vert ^\alpha \,,
	\end{align*}
	and hence $G$ bounds the first term in \eqref{eq:Gronwall}. The second term is estimated similarly by noting that $ Y^{-1}(t,\cdot)$ is the flow associated with the (backward) motion field $-v(1-t,\cdot)$.
	This proof can be further generalized to $C^0([0,1],C^{k,\alpha}(\overline\domain))$-norms provided that $m$ is sufficiently large.
\end{proof}
\begin{remark}\label{rem:DiffeoVeloRem}  
	Existence results and bounds analogous to those in \cref{thm:DiffeoVelo} hold when replacing $\mathcal V$ by $C^{1,\alpha}(\overline{\domain})$  with zero boundary condition \cite[Chapter~8]{Younes2010}.
	Furthermore, the mapping $v \to Y^v$ is Lipschitz continuous in $v$, i.e.
	\[
	\Vert Y^v(t,\cdot)-Y^{\tilde v}(t,\cdot)\Vert_{C^0(\overline\domain)}\leq(1+C\exp(C))\int_{0}^{t} \Vert v(s,\cdot) - \tilde v(s,\cdot)\Vert_{C^0(\overline\domain)} \dx s\,,
	\]
	where $C=\int_{0}^{t}\Vert v(s,\cdot)\Vert_{C^1(\overline\domain)} \dx s$, cf.~\cite[(8.16)]{Younes2010}.
\end{remark}

\paragraph{Metamorphosis} 
The metamorphosis model can be regarded as a generalization of the flow of diffeomorphisms model,
in which the brightness constancy assumption is replaced by a quadratic penalization of the material derivative, which in particular allows for intensity modulations along the trajectories.
Thus, as a first attempt the metric and the path energy in the metamorphosis model associated with the family of images $(\image(t))_{t\in[0,1]}\colon\overline\domain\rightarrow\R^n$ and a penalization parameter~$\delta>0$ are defined as follows
\begin{equation}
g(\dot\image,\dot\image)=\min_{v:\overline\domain\rightarrow\R^n}\int_\domain L[v,v]+\regparam
\left(\frac{D}{\partial t}\image\right)^2\dx x\,,\qquad
\boldsymbol{\mathcal{E}}(\image)=\int_0^1g(\dot\image(t),\dot\image(t))\dx t\,.
\label{eq:firstMetamorphosis}
\end{equation}
Thus, the flow of diffeomorphisms model is formally the limiting case of the metamorphosis model for $\delta\to 0$.

However, there are two major problems related to~\eqref{eq:firstMetamorphosis}.
Clearly, in general paths in the space of images do not exhibit any smoothness properties---neither
in space nor in time. Thus, the evaluation of
the material derivative~$(\frac{D}{\partial t}\image)^2$
is not well-defined. Moreover, since different pairs of velocity fields~$v$ and
material derivatives~$\frac{D}{\partial t}\image$ can imply the same time derivative
of the image path~$\dot\image$, the restriction to equivalence classes of 
pairs~$(v,\frac{D}{\partial t}\image)$ is required, where two pairs are equivalent
if and only if they induce the same temporal change of the image path~$\dot\image$.

To tackle both problems, Trouv\'e and Younes~\cite{TrYo05a} proposed
a nonlinear geometric structure in the space of images~$L^2(\domain)\coloneqq L^2(\domain,\R)$.
In detail, for a given velocity field~$v\in L^2((0,1), \mathcal V)$ 
and an image path~$\image\in L^2((0,1),L^2(\domain))$
the material derivative is replaced by the function $Z\in L^2((0,1),L^2(\domain))$
known as the \emph{weak material derivative}, which is uniquely determined by
\[
\int_0^1\int_\domain\eta Z\dx x\dx t=-\int_0^1\int_\domain(\partial_t\eta+\div(v\eta))\image\dx x\dx t
\]
for $\eta\in C^{\infty}_c((0,1)\times\domain)$.
Moreover, for all $\image\in L^2(\domain)$ the associated \emph{tangent space $T_\image L^2(\domain)$}
is defined as $T_\image L^2(\domain)=\{\image\}\times W/N_{\image}$, where $W= \mathcal V \times L^2(\domain)$ and
\[
N_{\image}=\left\{w=(v,Z)\in W:\int_\domain Z\eta+\image\div(\eta v)\dx x=0
\ \forall\eta\in C^{\infty}_c(\domain)\right\}\,.
\]
As usual, the associated \emph{tangent bundle} is given by
$TL^2(\domain)=\bigcup_{\image\in L^2(\domain)}T_{\image}L^2(\domain)$.

Then, following Trouv\'e and Younes, a \emph{regular path in the space of images}
(denoted by $\image\in H^1([0,1],L^2(\domain))$) is a curve $\image\in C^0([0,1],L^2(\domain))$ 
such that there exists a measurable path $\gamma\colon[0,1]\rightarrow TL^2(\domain)$ with 
bounded $L^2$-norm in space and time and $\pi(\gamma)=\image$,
where $\pi(\image,\overline{(v,Z)})=\image$ refers to the projection onto the image manifold
and $(\image,\overline{(v,Z)})$ denotes the equivalence class, such that
\[
-\int_0^1\int_\domain\image\partial_t\eta\dx x\dx t=\int_0^1\int_\domain Z\eta+\image\div(\eta v)\dx x\dx t
\]
for all $\eta\in C^{\infty}_c((0,1)\times\domain)$.
In this paper, we use the Lagrange formulation of this equation.
Let $Y$ be the coordinate transform given by \eqref{eq:contFlow}, then according to~\cite{TrYo05a} the weak material derivative is equivalently determined by the following integral equation 
\[
\image(t,Y(t,\cdot))-\image(s,Y(s,\cdot)) = \displaystyle\int_t^s Z(r,Y(r,\cdot))\dx r
\]
for all $s,t\in[0,1]$. This can be considered as a Lagrangian version of the classical material derivative.
Finally, if we assume the $\mathcal V$-coercivity
of the operator~$L$, then the \emph{path energy in the metamorphosis model}
for a regular path $\image \in H^1([0,1],L^2(\domain))$ is defined as
\begin{equation}
\boldsymbol{\mathcal{E}}(\image)=\int_0^1\inf_{\overline{(v,Z)}\in T_{\image(t)} L^2(\domain)}\int_\domain L[v,v]+\regparam Z^2\dx x\dx t\,.
\label{eq:DefinitionPathenergy}
\end{equation}
The existence of energy minimizing paths in the space of images (known as \emph{geodesic curves}), i.e.~solutions of the boundary value problem
\[
\min\{\boldsymbol{\mathcal{E}}(\tilde\image):\ \tilde\image\in H^1([0,1],L^2(\domain)),\
\tilde\image(0)=\image_A,\ \tilde\image(1)=\image_B\}
\]
for fixed images $\image_A,\image_B\in L^2(\domain)$, is proven in \cite{TrYo05a}.
In addition, one can prove the existence of minimizing $\overline{(v,Z)}\in T_{\image(t)} L^2(\domain)$.


We remark that all results of this paper can be easily generalized to the space of multichannel or color images~$L^2(\domain,\R^C)$ for $C\geq 2$
color channels with minor modifications.

\subsection{Manifold-valued time discrete metamorphosis model}
Now, we pick up the time discrete metamorphosis
model for manifold-valued images, for which the Mosco--convergence is studied in this paper.
The model itself was thoroughly analyzed in \cite{NPS17} and extends the variational time discretization of the 
classical metamorphosis model proposed in \cite{BeEf14}.

Fix $\gamma,\delta,\varepsilon>0$ and $m>1+\frac{n}{2}$, and let $\Hadamard$ be any finite-dimensional Hadamard manifold.
For two manifold-valued images~$\image,\tilde\image\in L^2(\domain,\Hadamard)$
and an \emph{admissible deformation}
\[
\varphi\in
\mathcal{A}_\varepsilon=\left\{\varphi\in H^m(\domain,\domain):\det D\varphi>\varepsilon\text{ in }\domain,
\varphi=\Id\text{ on }\partial\domain\right\}\,,
\]
the \emph{time discrete energy} for pairs of images is defined as 
\[
\boldsymbol{R}(\image,\tilde\image)
=\inf_{\varphi\in\mathcal{A}_\varepsilon}\boldsymbol{R}(\image,\tilde\image,\varphi)\,,
\]
where
\begin{equation}\label{eq:PairwiseEnergy}
\boldsymbol{R}(\image,\tilde\image,\varphi)=
\int_\domain \energyDensity(D\varphi(x))+\gamma\lVert D^m\varphi(x)\rVert^2\dx x
+\regparam\dist_2^2(\image,\tilde\image\circ \varphi)
\end{equation}
for an \emph{elastic energy density~$\energyDensity$}.
Here, $\dist_2^2(\cdot, \cdot)$ replaces the squared $L^2$-norm  in the time discrete metamorphosis model.
The energy~$\boldsymbol{R}$ can be considered as a numerically feasible
approximation of the squared Riemannian distance in the underlying image space~\cite{RuWi12b}.
Throughout this paper, we assume that $\energyDensity$ satisfies the following conditions:
\begin{enumerate}[label=(W\arabic*)]
	\item\label{W1}
	$\energyDensity\in C^4(\mathrm{GL}^+(n),\R_0^+)$ is polyconvex.
	\item\label{W2}
	There exist constants $C_{\energyDensity,1},C_{\energyDensity,2},r_\energyDensity>0$ such that
	for all $A\in\mathrm{GL}^+(n)$ the following growth estimates hold true:
	\begin{align}
	\energyDensity(A)&\geq C_{\energyDensity,1}\Vert A^\mathrm{sym}-\onesymbol\Vert ^2\,,
	&\text{if }\Vert A-\onesymbol\Vert <r_\energyDensity\,,
	\label{eq:energy3}\\
	\energyDensity(A)&\geq C_{\energyDensity,2}\,,
	&\text{if }\Vert A-\onesymbol\Vert \geq r_\energyDensity\,.
	\label{eq:energy4}
	\end{align}
	\item\label{W3}
	The energy density admits the following representation at $\onesymbol$:
	\begin{align}
	\energyDensity(\onesymbol)&=0\,,\quad
	D\energyDensity(\onesymbol)=0\,,\label{eq:energy1}\\
	\frac12 D^2\energyDensity(\onesymbol)(A,A)
	&=\frac{\lambda}{2}(\tr A)^2+\mu\tr\left(\left(A^\mathrm{sym}\right)^2\right)\,.
	\label{eq:energy2}
	\end{align}
\end{enumerate}
The assumption~\ref{W1} is required for the lower semi-continuity of the energy functional.
Furthermore, \ref{W2} enforces the convergence of the optimal deformations to the identity in the limit~$K\rightarrow\infty$, where $K$ denotes the number of time steps of
our time discrete model to be defined next in \eqref{eq:d_path}.
Finally, \ref{W3} ensures the compatibility of~$\energyDensity$ with
the elliptic operator~$L$ (cf.~\eqref{eq:ellipticOperator}).
Note that \ref{W1} and \ref{W3} are identical to \cite[(W1) and (W3)]{BeEf14}.
We recall that in \cite[(W2)]{BeEf14} a growth estimate of the form
\begin{equation}
\energyDensity(A)\geq C(\det A)^{-s}-C
\label{eq:oldW2}
\end{equation}
for $s>n-1$ and a positive constant~$C$ instead of \ref{W2} is assumed.
This modification additionally requires essentially bounded images in order to ensure
that the deformations are homeomorphic. 
However, in order to use the Hadamard space of square-integrable images, we have to use \ref{W2} instead,
which in particular results in diffeomorphic deformations.

The \emph{time discrete path energy} for $K+1$ images
$\boldsymbol{\image}=(\image_0,\ldots,\image_K)\in(L^2(\domain,\Hadamard))^{K+1}$, $K\geq 2$,
is defined as the weighted sum of the discrete energies $\boldsymbol{R}$ evaluated at consecutive images, i.e.
\begin{equation}
\boldsymbol{J}_K (\boldsymbol{\image}) 
\coloneqq
\inf_{\boldsymbol{\varphi} \coloneqq (\varphi_1,\dots,\varphi_K) \in(\mathcal{A}_\varepsilon)^K}
\left\{
\boldsymbol{J}_K(\boldsymbol{\image}, \boldsymbol{\varphi}) 
\coloneqq
K\sum_{k = 1}^{K}\boldsymbol{R}(\image_{k-1},\image_k,\varphi_k) 
\right\}\,.
\label{eq:d_path}
\end{equation}
The scaling factor~$K$ in \eqref{eq:d_path} is a natural choice in this time discrete geodesic calculus.
Indeed, if we sample a continuous path~$y\colon[0,1]\to\manifold$ on a Riemannian manifold~$(\manifold,\metric)$ at~$t_{K,k}=\frac{k}{K}$
for $k=0,\ldots,K$, we obtain from Jensen's inequality
\[
\sum_{k=1}^K\dist(y(t_{K,k-1}),y(t_{K,k}))^2\leq\sum_{k=1}^K\frac{1}{K}\int_{t_{K,k-1}}^{t_{K,k}}\metric_{y(t)}(\dot y(t),\dot y(t))\dx t=\frac{1}{K}\int_0^1\metric_{y(t)}(\dot y(t),\dot y(t))\dx t\,.
\]
A more rigorous justification is given in~\cite{RuWi12b}.

For two fixed images $\image_A=\image_0,\, \image_B=\image_K\in L^2(\domain,\Hadamard)$
a $(K+1)$-tuple $\boldsymbol{\image}=(\image_0,\ldots,\image_K)\in(L^2(\domain,\Hadamard))^{K+1}$
is called a \emph{discrete geodesic curve} if 
\[
\boldsymbol{J}_K (\boldsymbol{\image})\leq
\boldsymbol{J}_K ((\image_0,\tilde\image_1,\ldots,\tilde\image_{K-1},\image_K))
\]
for all $(\tilde\image_1,\ldots,\tilde\image_{K-1})\in(L^2(\domain,\Hadamard))^{K-1}$.
The existence of discrete geodesic curves has been shown in \cite[Section~3]{NPS17} using \ref{W1} and the properties of the deformation set $(\mathcal A_\varepsilon)^K$.
Note that in general neither
the discrete geodesic curve nor the associated set of deformations is uniquely determined.
The Mosco--convergence of a temporal extension of $\boldsymbol{J}_K$
to $\boldsymbol{\mathcal{E}}$ in the Euclidean case was proven in \cite{BeEf14}.
\begin{figure}[ht]
	\centering
	\begin{tikzpicture}
	\node at (0,0) {\includegraphics[width=0.35\linewidth]{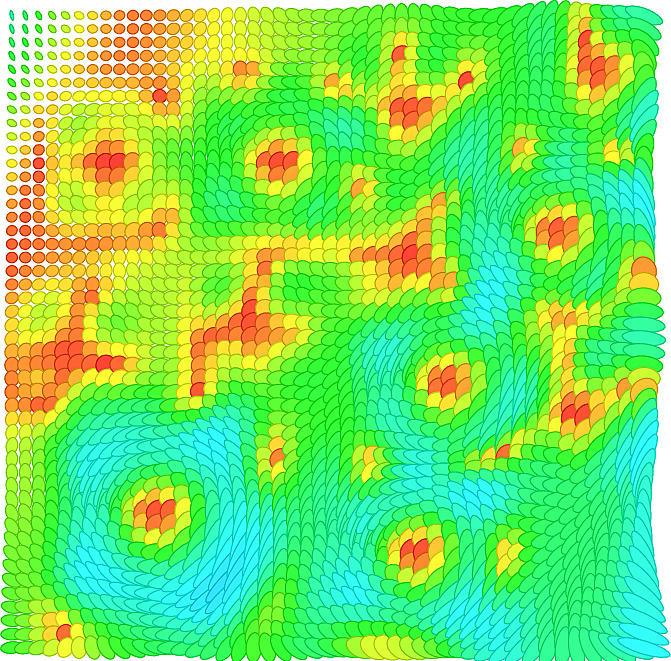}};
	\node at (8,0) {\includegraphics[width=0.35\linewidth]{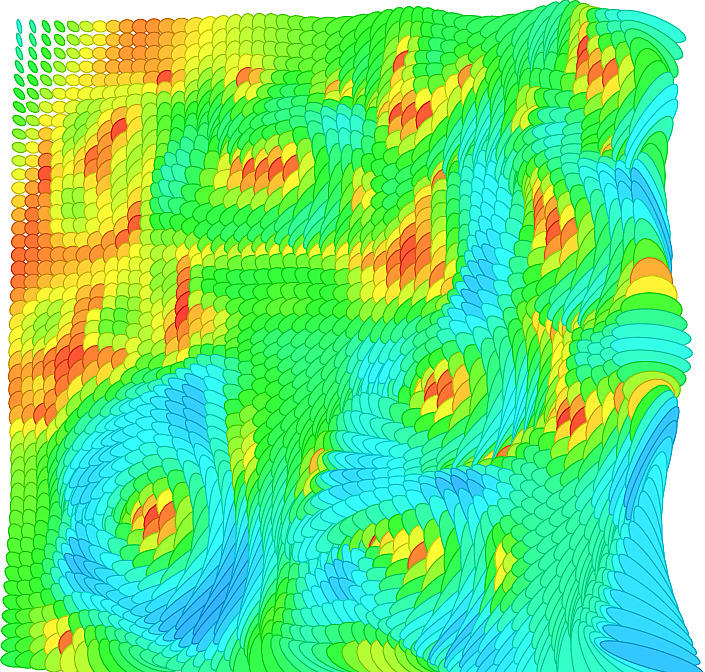}};
	\end{tikzpicture}
	\caption{The synthesized input images used in all computations, where the diffusion tensors are 
		visualized as ellipsoids color-coded with respect to the geometric anisotropy.
	}
	\label{fig:inputImages}
\end{figure}
\Cref{fig:input} shows different discrete geodesic paths for $K=4,\,8,\,16$ connecting two synthesized input images of symmetric and positive definite matrices in~$\R^2$ visualized in~\Cref{fig:inputImages}.
Here, the colors quantify the geodesic anisotropy index and the eigenvectors of the matrices correspond to the principle axes of the ellipses (for further details of the visualization we refer the reader to~\cite{MoBa06}).
For all computations, a finite difference discretization on staggered grids proposed in~\cite{NPS17} was used.
In particular, one experimentally observes an indication of convergence for increasing~$K$.
\begin{figure}[!p]
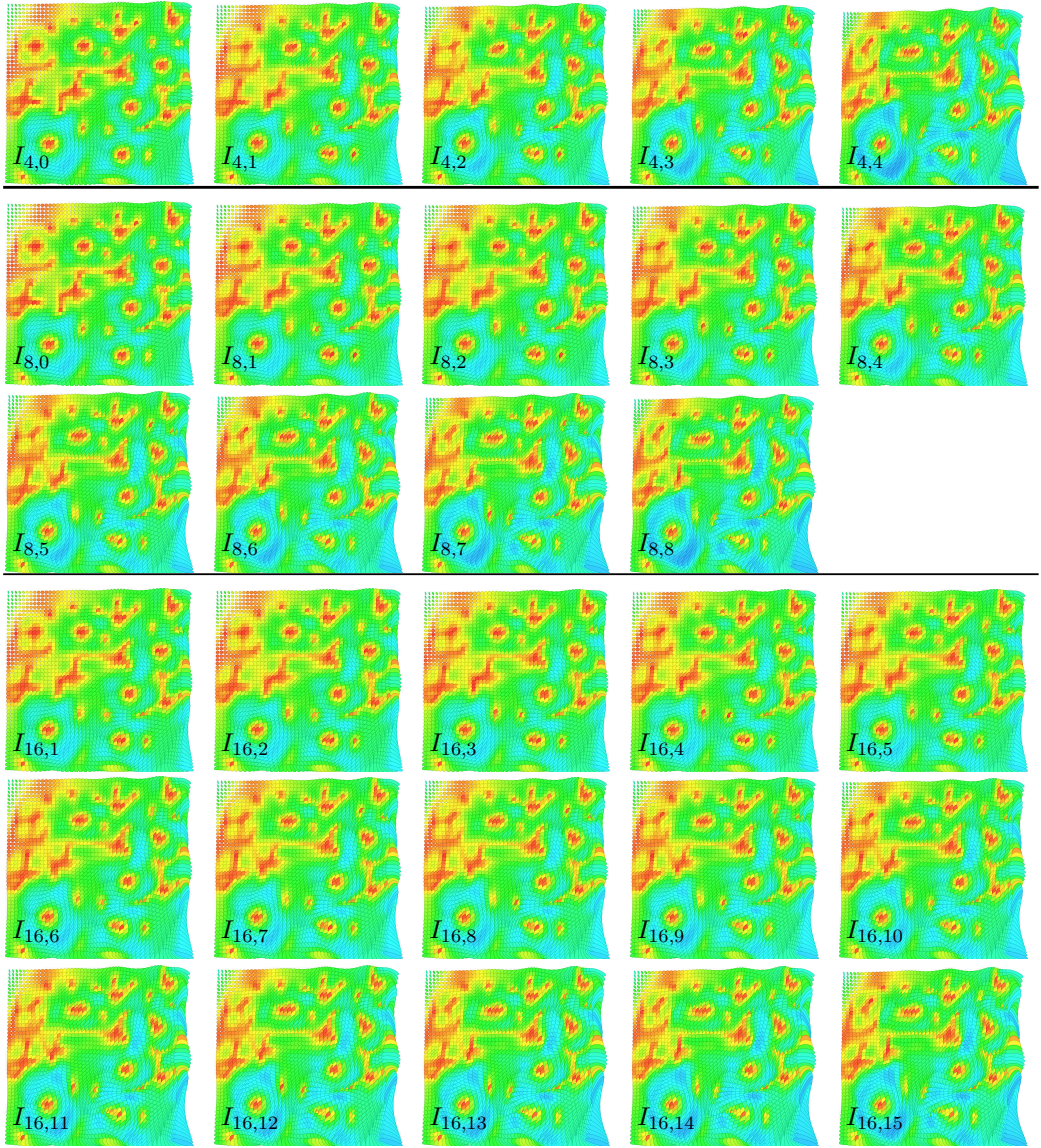

	\centering
	\resizebox{.95\linewidth}{!}{
		\begin{tikzpicture}
		
		\begin{scope}[shift={(0,0)}]
		\edef\currentCnt{0}
		\foreach \j in {1,...,5}{
			\node[anchor=south west] at (\currentCnt,0) {\includegraphics[width=0.19\linewidth]{Images/NewImages/Result_5/Morph_\j.jpg}};
			\FPeval{\jj}{clip(\j-1)}
			\node[anchor=south west] at (\currentCnt+.1,.2) {$\image_{4,\jj}$};
			\pgfmathparse{\currentCnt+3.1}
			\xdef\currentCnt{\pgfmathresult}
		}
		\end{scope}
		
		\begin{scope}[shift={(0,-3)}]
		\edef\currentCnt{0}
		\foreach \j in {1,...,5}{
			\node[anchor=south west] at (\currentCnt,0) {\includegraphics[width=0.19\linewidth]{Images/NewImages/Result_9/Morph_\j.jpg}};
			\FPeval{\jj}{clip(\j-1)}
			\node[anchor=south west] at (\currentCnt+.1,.2) {$\image_{8,\jj}$};
			\pgfmathparse{\currentCnt+3.1}
			\xdef\currentCnt{\pgfmathresult}
		}
		\end{scope}
		
		\begin{scope}[shift={(0,-5.8)}]
		\edef\currentCnt{0}
		\foreach \j in {6,...,9}{
			\node[anchor=south west] at (\currentCnt,0) {\includegraphics[width=0.19\linewidth]{Images/NewImages/Result_9/Morph_\j.jpg}};
			\FPeval{\jj}{clip(\j-1)}
			\node[anchor=south west] at (\currentCnt+.1,.2) {$\image_{8,\jj}$};
			\pgfmathparse{\currentCnt+3.1}
			\xdef\currentCnt{\pgfmathresult}
		}
		\end{scope}
		
		\begin{scope}[shift={(0,-8.8)}]
		\edef\currentCnt{0}
		\foreach \j in {2,...,6}{
			\node[anchor=south west] at (\currentCnt,0) {\includegraphics[width=0.19\linewidth]{Images/NewImages/Result_17/Morph_\j.jpg}};
			\FPeval{\jj}{clip(\j-1)}
			\node[anchor=south west] at (\currentCnt+.1,.2) {$\image_{16,\jj}$};
			\pgfmathparse{\currentCnt+3.1}
			\xdef\currentCnt{\pgfmathresult}
		}
		\end{scope}
		
		\begin{scope}[shift={(0,-11.6)}]
		\edef\currentCnt{0}
		\foreach \j in {7,...,11}{
			\node[anchor=south west] at (\currentCnt,0) {\includegraphics[width=0.19\linewidth]{Images/NewImages/Result_17/Morph_\j.jpg}};
			\FPeval{\jj}{clip(\j-1)}
			\node[anchor=south west] at (\currentCnt+.1,.2) {$\image_{16,\jj}$};
			\pgfmathparse{\currentCnt+3.1}
			\xdef\currentCnt{\pgfmathresult}
		}
		\end{scope}
		
		\begin{scope}[shift={(0,-14.4)}]
		\edef\currentCnt{0}
		\foreach \j in {12,...,16}{
			\node[anchor=south west] at (\currentCnt,0) {\includegraphics[width=0.19\linewidth]{Images/NewImages/Result_17/Morph_\j.jpg}};
			\FPeval{\jj}{clip(\j-1)}
			\node[anchor=south west] at (\currentCnt+.1,.2) {$\image_{16,\jj}$};
			\pgfmathparse{\currentCnt+3.1}
			\xdef\currentCnt{\pgfmathresult}
		}
		\end{scope}
		
		\draw [very thick] (0.1,0.1) -- (15.5,0.1);
		\draw [very thick] (0.1,-5.7) -- (15.5,-5.7);
		
		\end{tikzpicture}
	}
	\caption{Time discrete geodesic paths for $K=4,8,16$ (the input images for $K=16$ are not depicted).
		Note that the images $\image_{4,i}$, $\image_{8,2 i}$ and $\image_{16,4 i}$ reflect the increasing similarity expected for larger~$K$ in correspondence to the convergence result stated in this paper.
	}
	\label{fig:input}
\end{figure}

\section{Manifold-valued metamorphosis model}\label{sec:manifoldMetamorphosis}
In this section, we propose a (time continuous) metamorphosis energy functional~$\boldsymbol{\mathcal{J}}$
for manifold-valued images in~$L^2(\domain,\Hadamard)$, where $\Hadamard$ is a finite-dimensional Hadamard manifold.
This functional substantially differs from the straightforward generalization 
\begin{equation*}
\inf_{(v,Z)\in\sol(\image)}\int_0^1\int_\domain L[v,v]+\regparam \metric^\Hadamard_\image(Z,Z)\dx x \dx t\,.
\end{equation*}
of the classical metamorphosis functional in \eqref{eq:firstMetamorphosis}, where $\metric^\Hadamard_\image$ is the Hadamard metric at position~$\image$ on~$\Hadamard$. 
Indeed, a generalization of the weak notion of the material derivative as a tangent vector $Z(x)\in T_{\image(x)}\Hadamard$
on the Hadamard manifold via a defining equation in the context of a corresponding weak formulation is technically involved.
For a given image curve~$t\mapsto\image(t,Y(t,\cdot))$, the associated tangential vectors at different times are in general contained in different tangent spaces and compactness of the metric $g^\Hadamard_\image$ 
in the base point on the Hadamard manifold is not to be expected for sequences of paths in $L^2(\domain,\Hadamard)$. 
Hence, we propose a relaxation via an inequality relating distances between images along the motion path and an associated scalar material derivative $z$, where 
$z=\Vert Z \Vert$ in the Euclidean case of images in $L^2(\domain, \R^C)$ . 
At first, this relaxed definition of the material derivative via the variational inequality~\eqref{eq:ODESys2} avoids the above technical difficulty in the definition.
Furthermore, this relaxed formulation will turn out to be suitable for lower semi-continuity considerations
which are needed to identify this energy in~\cref{sec:Mosco} as the Mosco--limit of the above time discrete path energy and to establish existence of geodesic paths for the novel metamorphosis model.

Furthermore, we prove the equivalence of this novel energy functional with the classical metamorphosis model for $\R^C$-valued images, where
the scalar  material derivative coincides with the norm of the classical material derivative.

The \emph{manifold-valued metamorphosis
	energy functional~$\boldsymbol{\mathcal{J}}\colon L^2((0,1) \times\domain,\Hadamard)\to[0,\infty]$} is defined 
as follows
\begin{equation}\label{eq:c_path}
\boldsymbol{\mathcal{J}} (\image) 
\coloneqq\inf_{(v,z)\in\sol(\image)}\int_0^1\int_\domain L[v,v]+\regparam z^2\dx x\dx t\,.
\end{equation}
Here, $\sol(\image)$ is the set of pairs
$(v,z) \in  L^2((0, 1), \mathcal V) \times L^2((0,1), L^2(\domain))$
such that the flow $Y$ defined by 
\begin{equation} \label{eq:ODESys1}
\begin{array}{rll}
\displaystyle{\frac{\dx}{\dx t}}Y(t,x)&=\,\,v(t,Y(t,x))& \text{for } (t,x) \in [0,1]\times\domain\,,\\[.5em]
Y(0,x)&=\,\,x & \text{for } x\in \domain
\end{array}
\end{equation}
satisfies for all $t<s \in [0,1]$ the inequality
\begin{equation}\label{eq:ODESys2}
d\big(\image(t,Y(t,\cdot)),\image(s,Y(s,\cdot))\big) \leq\displaystyle{\int_t^s}z(r,Y(r,\cdot))\dx r\,.
\end{equation}
Let us verify the equivalence of this new relaxed model with the classical metamorphosis model for $\R^C$-valued images.
In the classical model, the ($C$-dimensional) material derivative $Z$ is defined via the equation
\begin{equation}\label{eq:tildez}
\image(t,Y(t,\cdot))-\image(s,Y(s,\cdot)) = \displaystyle\int_t^s Z(r,Y(r,\cdot))\dx r
\end{equation}
for all $t<s \in [0,1]$, whereas the scalar material derivative $z$ obeys the inequality 
\begin{equation}\label{eq:z}
\Vert\image(t,Y(t,\cdot))-\image(s,Y(s,\cdot))\Vert\leq\displaystyle\int_t^s z(r,Y(r,\cdot))\dx r\,.
\end{equation}
In fact, the equivalence is already implied by the following proposition, which in particular
proves that the manifold-valued metamorphosis energy~\eqref{eq:c_path} coincides with the metamorphosis energy functional~\eqref{eq:DefinitionPathenergy}
in the case of $\R^C$-valued images.
\begin{proposition}\label{prop:equivalenceMM}
	For every $z$ fulfilling \eqref{eq:z} there exists a $Z$ fulfilling \eqref{eq:tildez} with $z\geq\Vert Z\Vert$.
	Conversely, for every $Z$ fulfilling \eqref{eq:tildez} there exists a $z$ fulfilling \eqref{eq:z} with $z=\Vert Z\Vert$.
\end{proposition}
\begin{proof}
	For given $Z$ the result follows from the triangle inequality by choosing $z = \Vert Z \Vert$.
	To prove the converse, let $z$ solve \eqref{eq:z}.
	Taking the $L^2$-norm on both sides implies
	\begin{align*}
	\Vert\image(t,Y(t,\cdot))-\image(s,Y(s,\cdot))\Vert_{L^2(\domain)}
	\leq \int_s^t \Vert z(r,Y(r,\cdot)) \Vert_{L^2(\domain)}\dx r\,,
	\end{align*}
	i.e.~the function~$t\mapsto\image(t,Y(t,x))$ is $AC^2([0,1],L^2(\domain))$ in the sense of \cite[Definition~1.1.1]{Ambrosio}.
	Using \cite[Remark~1.1.3]{Ambrosio} one can additionally infer the a.e.~differentiability with derivative $\widehat Z\in L^2((0,1),L^2(\domain))$ such that 
	\[
	\image(t,Y(t,x))-\image(0,Y(0,x))=\int_0^t\widehat Z(r,x)\dx r = \int_0^t Z(r,Y(r,x))\dx r
	\]
	with $Z(r,x) \coloneqq\widehat Z(r,X(r,x))$. Here, $X(r,\cdot)$ is the spatial inverse of $Y(r,\cdot)$, which exists due to \cref{thm:DiffeoVelo}.
	Now set
	\[
	B=\left\{
	(r,x)\in[0,1]\times\domain\;:\;z(r,Y(r,x)) < \Vert Z(r,Y(r,x))\Vert\right\}
	\]
	and assume that the Lebesgue measure of $B$ is strictly positive.
	Note that $B$ can be approximated with finite unions of disjoint semi-open cuboids \cite[Theorem 1.4]{SteSha05}.
	By taking into account \cite[Theorem~1.1.2/Remark~1.1.3]{Ambrosio}, for every such 
	cuboid $[t_1,t_2)\times D\subset[0,1]\times\domain$ we obtain
	\[
	\int_{t_1}^{t_2} \int_D \Vert Z(t,Y(t,x))\Vert^2 \dx x\dx t \leq \int_{t_1}^{t_2}  \int_D z(t,Y(t,x))^2 \dx  x \dx t\,.
	\]
	Combining this estimate with the dominated convergence theorem we conclude
	\[
	\int_B \Vert Z(t,Y(t,x)) \Vert^2 \dx x \dx t \leq \int_B z(t,Y(t,x))^2 \dx x \dx t\,.
	\]
	This yields a contradiction to the definition of the set~$B$. Hence, $z\geq \Vert Z \Vert$ a.e.~in $t$ and $x$.
\end{proof}

\section{Temporal extension operators}\label{sec:Extension}
In this section, temporal extensions of all relevant quantities required for the convergence proof of the time discrete metamorphosis
are proposed, which in particular allows an explicit solution to the optimality conditions~\eqref{eq:ODESys1} and ~\eqref{eq:ODESys2}.
We remark that the subsequent construction is similar to \cite{BeEf14} with two major modifications, namely
the definitions of the interpolated image sequence~\eqref{eq:DefUk} and the weak material derivative~\eqref{eq:materialDerivativeDef},
which are related to the manifold structure.

For fixed $K \in \N$, let a discrete image path $\boldsymbol{\image}_K =(\image_{K,0}, \dots, \image_{K,K}) \in L^2(\domain, \Hadamard)^{K+1}$ be given.
Existence of the corresponding optimal deformations $\boldsymbol \varphi_K = ({\varphi}_{K,1}, \dots, {\varphi}_{K,K})\in(\mathcal{A}_\varepsilon)^K$ satisfying \eqref{eq:d_path} is proven in \cite[Section~3]{NPS17}.
We refer to $\tau =  K^{-1}$ as the \emph{time step size}
and the image~$\image_{K,k}$ is associated with the \emph{time step $t_{K,k} = k\tau$}, $k=0,\dots,K$.
For $k=1,\dots,K$, we define the \emph{discrete transport map $y_{K,k} \colon [t_{K,k-1}, t_{K,k}] \times \overline{\domain} \to \overline{\domain}$} as 
\begin{equation}
y_{K,k}(t,x) \coloneqq x + (t - t_{K,k-1})K(\varphi_{K,k}(x) - x)\,.
\end{equation}
If
\begin{equation}\label{eq:AssumptionDiffeo}
\max_{k=1,\ldots,K}\Vert  \varphi_{K,k} - \Id \Vert_{C^{1,\alpha}(\overline{\domain})} < 1\,,
\end{equation}
we can use \cite[Theorem 5.5-1/Theorem 5.5-2]{Cia88}
to infer that $\det(Dy_{K,k}(t,\cdot))>0$ holds
and that $y_{K,k}(t,\cdot)$ is invertible with inverse $x_{K,k}(t,\cdot)$.
The validity of this assumption is proven below and is tacitly assumed for all further considerations.

Next, the \emph{extension operator
	$\image^{\mathrm{ext}}_K \colon L^2(\domain,\Hadamard)^{K+1}\times(\mathcal{A}_\varepsilon)^K\to L^2([0,1],L^2(\domain,\Hadamard))$},
is defined for $t\in[t_{K,k-1}, t_{K,k})$ and a.e.~$x\in\domain$ by
\begin{equation}\label{eq:DefUk}
\image^{\mathrm{ext}}_K(\boldsymbol{I}_K,\boldsymbol{\varphi}_K)(t,x)
\coloneqq \gamma_{\image_{K,k-1}(x_{K,k}(t,x)),\image_{K,k}\circ\varphi_{K,k}(x_{K,k}(t,x))}(K(t-t_{K,k-1}))\,.
\end{equation}
Here, 
$\gamma_{\image_{K,k-1}(x_{K,k}(t,x)),\image_{K,k}\circ\varphi_{K,k}(x_{K,k}(t,x))}(K(t-t_{K,k-1}))$ is a 
point on the geodesic between  $I_{K,k-1}(x_{K,k}(t,x))$ and $I_{K,k}\circ\varphi_{K,k}(x_{K,k}(t,x))$ on the manifold $\Hadamard$.
Thus, $\image^{\mathrm{ext}}_K$ uniquely describes for given $\boldsymbol{I}_K$ and $\boldsymbol{\varphi}_K$ a blending in the geodesic sense
along the transport path governed by $y_{K,k}$.

In what follows, we set $w_{K,k}=K(\varphi_{K,k}-\Id)$ and define
the \emph{piecewise constant (in time) velocity}
$\velocityConstant_K=\velocityConstant_K(\boldsymbol \varphi_K) \in L^2((0,1),\mathcal V)$ as
\[
\velocityConstant_K(\boldsymbol \varphi_K)\big\vert _{[t_{K,k-1}, t_{K,k})} \coloneqq w_{K,k}\,.
\]
Furthermore, we define the \emph{discrete velocity field}
$\velocityPullback_K\colon \mathcal V^K \to L^2((0,1), C^{1,\alpha}(\overline{\domain}))$,
\[
\velocityPullback_K(\boldsymbol \varphi_K)(t,x) \coloneqq K(\varphi_{K,k} - \Id )(x_{K,k}(t,x))
\]
for $t\in[t_{K,k-1}, t_{K,k})$ and a.e.~$x\in\domain$, which is constant along time discrete paths.

Note that the extension operator $v_K$ merely admits a $C^{1,\alpha}$-regularity. 
To see this, we note that the composition of $f\in C^{1,\alpha}(\overline{\domain})$ and $g\in C^{1,\alpha}(\overline{\domain},\overline{\domain})$
is in $C^{1,\alpha}(\overline{\domain})$ and the estimate
\[
\Vert f \circ g \Vert_{C^{1,\alpha}(\overline{\domain})} \leq \Vert f \Vert_{C^{1}(\overline{\domain})}(1 + \Vert g \Vert_{C^{1}(\overline{\domain})}) + [Df \circ g\,Dg]_{\alpha}\leq C\Vert f \Vert_{C^{1,\alpha}(\overline{\domain})}\left(1+\Vert g \Vert_{C^{1,\alpha}(\overline{\domain})}\right)^2,
\]
follows from \cite[Proposition~1.2.4 and Proposition~1.2.7]{Fio16}, where $[\cdot]_\alpha$ denotes the H\"older constant.
Taking into account \cite[Theorem 2.1]{BHS2005},
we infer that $x_{K,k}(t,\cdot) \in C^{1,\alpha}(\overline{\domain})$ and
\[
D(x_{K,k}(t,\cdot)) =  K^{-1}\mathrm{Inv} \left( K^{-1}\onesymbol + (t - t_{K,k-1})(D\varphi_{K,k} - \onesymbol) (x_{K,k}(t,\cdot))\right)\,,
\]
where $\mathrm{Inv}\colon GL(n) \to GL(n)$ denotes the smooth inversion operator.
Since $\domain$ is bounded and $x_{K,k}(t,\cdot)$ is a diffeomorphism, we get
\begin{equation}
\Vert x_{K,k}(t,\cdot) \Vert_{C^{1,\alpha}(\overline{\domain})} \leq C + \Vert Dx_{K,k}(t,\cdot) \Vert_{C^{0,\alpha}(\overline{\domain})}
\leq C\Bigl(1 +  K^{-1}\max_{k=1,\ldots,K}\Vert  \varphi_{K,k} - \Id \Vert_{C^{1,\alpha}(\overline{\domain})}\Bigl)\,,
\label{eq:BoundInverse}
\end{equation}
where the mean value theorem is applied to $x_{K,k}$.
This implies that $\velocityPullback_K(t,\cdot) \in C^{1,\alpha}(\overline{\domain})$ and
\begin{equation}
\Vert \velocityPullback_K (t,\cdot) \Vert_{C^{1,\alpha}(\overline{\domain})}\leq C\Vert\velocityConstant_K(t,\cdot)
\Vert_{C^{1,\alpha}(\overline{\domain})}\left(1+K^{-1}\Vert\velocityConstant_K(t,\cdot)\Vert_{C^{1,\alpha}(\overline{\domain})}\right)^2\,.
\label{eq:velocityEstimate}
\end{equation}

As a last preparatory step, we define the \emph{discrete path $Y_K \colon [0,1] \times \overline{\domain} \to\overline{\domain}$}, which is the concatenation of all small diffeomorphisms $y_{K,k}$ along the motion path.
In detail, the mapping is defined for $t \in [0, t_{K,1}]$ by
$Y_K(t,x) \coloneqq y_{K,1}(t,x)$
and then recursively for $k=2,\dots,K$ and $t \in (t_{K,k-1}, t_{K,k}]$ by
\[
Y_K(t,x) \coloneqq y_{K,k}\left(t,Y_K(t_{K,k-1},x)\right)
\]
for all $x \in \domain$. The spatial inverse of $Y_K$ is denoted by $X_K$.
Finally, we define the \emph{material derivative}
$z_K \in L^2((0,1), L^2(\domain))$ for $t \in [t_{K,k-1}, t_{K,k})$ as
\begin{equation}
z_K(t,x) \coloneqq K d\big(\image_{K,k-1}(x_{K,k}(t,x)),\image_{K,k}\circ \varphi_{K,k} (x_{K,k}(t,x))\big)\,.
\label{eq:materialDerivativeDef}
\end{equation}

In the following proposition, we prove that the temporal extensions of the images, the velocities, the material derivatives and
the discrete paths are indeed an admissible point for the problem, i.e.~they satisfy \eqref{eq:ODESys1} and \eqref{eq:ODESys2}.
\begin{proposition}[Admissible extension]\label{prop:admExtension}
	For $\boldsymbol{\image}_K\in L^2(\domain,\Hadamard)^{K+1}$ and deformations $\boldsymbol{\varphi}_K \in(\mathcal{A}_\varepsilon)^K$ satisfying \eqref{eq:AssumptionDiffeo},
	the tuple $(\image^{\mathrm{ext}}_K(\boldsymbol{\image}_K, \boldsymbol{\varphi}_K),\velocityPullback_K(\boldsymbol{\varphi}_K), Y_K, z_K)$ is a solution to \eqref{eq:ODESys1} and \eqref{eq:ODESys2}.
\end{proposition}
\begin{proof}
	By definition, we obtain $Y_K(0,x) = x$ for all $x \in \domain$. For $t \in [t_{K,k-1}, t_{K,k}]$ and $x \in \domain$ we get
	\begin{align*}
	\frac{\dx}{\dx t}Y_K(t,x)&=\frac{\dx}{\dx t}y_{K,k}(t,Y_K(t_{K,k-1},x))=K(\varphi_{K,k}-\Id)(Y_K(t_{K,k-1},x))\\&=\velocityPullback_K(\boldsymbol{\varphi}_K)(t,Y_K(t,x))\,.
	\end{align*}
	Therefore, $Y_K$ is a solution of \eqref{eq:ODESys1} in the weak sense according to \cref{rem:DiffeoVeloRem}.
	A short computation shows for $s\leq t\in [t_{K,k-1}, t_{K,k}]$ that
	\begin{align*}
	&d\Bigl(\image^{\mathrm{ext}}_K(\boldsymbol{\image}_K,\boldsymbol{\varphi}_K)(t,Y_K(t,x)), \image^{\mathrm{ext}}_K(\boldsymbol{\image}_K, \boldsymbol{\varphi}_K)(s,Y_K(s,x))\Bigr)\\
	=& d\Bigl(\gamma_{{\image_{K,k-1},\image_{K,k}\circ \varphi_{K,k}}}(K(t-t_{K,k-1}))(Y_K(t_{K,k-1},x)),\\ &\quad \quad \quad
	\gamma_{{\image_{K,k-1},\image_{K,k}\circ \varphi_{K,k}}}(K(s-t_{K,k-1}))(Y_K(t_{K,k-1},x))\Bigr)\\
	=&K(t-s)d\Bigl(\image_{K,k-1}(Y_K(t_{K,k-1},x)),\image_{K,k}\circ \varphi_{K,k}(Y_K(t_{K,k-1},x))\Bigr)\\
	\leq &\int_s^t z_K(r,Y_K(r,x)) \dx r\,.
	\end{align*}
	The first equation follows from the definition of the extension operator~\eqref{eq:DefUk},
	for the second equation we exploit the geodesic property~\eqref{eq:distancePropertyGeodesics}.
	Finally, the last inequality is implied by the definition of the weak material derivative~\eqref{eq:materialDerivativeDef}.
	If $s$ and $t$ are not in the same interval, we can use the triangle inequality multiple times, which concludes the proof.
\end{proof}
The next lemma allows us to bound the $H^m(\domain)$-norm of the displacements by a function solely depending on the energy~$\boldsymbol{R}$.
\begin{lemma}\label{lemm:growthControl}
	Under the assumptions~\ref{W1} and \ref{W2} there exists a continuous and mono\-tonically increasing function $\theta \colon \R^+_0\rightarrow\R^+_0$
	with $\theta(0)=0$ such that
	\[
	\Vert\varphi-\Id\Vert_{H^m(\domain)}\leq\theta\left(\boldsymbol{R}(\image,\tilde\image,\varphi)\right)
	\]
	for all $\image,\tilde\image\in L^2(\domain, \Hadamard)$ and all $\varphi\in\mathcal{A}_\varepsilon$.
	Furthermore, $\theta(x)\leq C(x+x^2)^\frac{1}{2}$ for a constant $C>0$.
\end{lemma}
\begin{proof}
	Set $\overline{\boldsymbol{R}}=\boldsymbol{R}(\image,\tilde\image,\varphi)$, which is defined in~\eqref{eq:PairwiseEnergy}.
	The Gagliardo--Nirenberg inequality~\cite{Nir1966} implies
	\begin{equation}
	\Vert\varphi-\Id\Vert_{H^m(\domain)}\leq C\left(\Vert\varphi-\Id\Vert_{L^2(\domain)}+\vert \varphi-\Id\vert _{H^m(\domain)}\right)\,.
	\label{eq:mGrowthDisplacement}
	\end{equation}
	The $H^m(\domain)$-seminorm of the displacement can be controlled as follows
	\begin{equation}
	\vert \varphi-\Id\vert _{H^m(\domain)}=\vert \varphi\vert _{H^m(\domain)}\leq\sqrt{\tfrac{\overline{\boldsymbol{R}}}{\gamma}}\,,
	\label{eq:displacementHigherOrderControl}
	\end{equation}
	which is implied by the definition of $\overline{\boldsymbol{R}}$.
	Since $\varphi\in H^m(\domain,\domain)$ implies $\Vert\varphi-\Id\Vert_{L^2(\domain)} \leq 2 \text{diam}(\Omega)$, this already shows for $\alpha\in(0,m-1-\frac{n}{2})$ that
	\begin{equation}
	\Vert\varphi-\Id\Vert_{C^{1,\alpha}(\overline\domain)}\leq C\Vert\varphi-\Id\Vert_{H^m(\domain)}\leq C+C\sqrt{\overline{\boldsymbol{R}}}\,.
	\label{eq:displacementControl}
	\end{equation}
	To control the lower order term appearing on the right-hand side of~\eqref{eq:mGrowthDisplacement}, we first define the set
	$\domain'=\{x\in\domain:\Vert D\varphi(x)-\Id\Vert <r_\energyDensity\}$.
	Then, by using \eqref{eq:energy3} and \eqref{eq:energy4}, we obtain
	\[
	\vert \domain\backslash\domain'\vert C_{\energyDensity,2}\leq\int_\domain\energyDensity(D\varphi)\dx x\leq\overline{\boldsymbol{R}}\,,
	\]
	which implies $\vert \domain\backslash\domain'\vert \leq\frac{\overline{\boldsymbol{R}}}{C_{\energyDensity,2}}$.
	Hence, by taking into account \eqref{eq:displacementControl}, we deduce
	\begin{align}
	\int_\domain\Vert (D\varphi)^\mathrm{sym}-\onesymbol\Vert ^2\dx x
	&=
	\int_{\domain'}\Vert (D\varphi)^\mathrm{sym}-\onesymbol\Vert ^2\dx x+\int_{\domain\backslash\domain'}\Vert (D\varphi)^\mathrm{sym}-\onesymbol\Vert ^2\dx x\notag\\
	&\leq
	\int_\domain\frac{\energyDensity(D\varphi)}{C_{\energyDensity,1}}\dx x+\vert \domain\backslash\domain'\vert \left(C+C\sqrt{\overline{\boldsymbol{R}}}\right)^2\notag\\
	&\leq\frac{\overline{\boldsymbol{R}}}{C_{\energyDensity,1}}+\frac{\overline{\boldsymbol{R}}}{C_{\energyDensity,2}}\left(C+C\overline{\boldsymbol{R}}\right)\,.
	\label{eq:convergenceLowerLTwo}
	\end{align}
	Thus, the lemma follows from \eqref{eq:mGrowthDisplacement}, where the first term is estimated by combining Korn's inequality with \eqref{eq:convergenceLowerLTwo}, and the second term is estimated using \eqref{eq:displacementHigherOrderControl}.
\end{proof}

\section{Mosco--Convergence of time discrete geodesic paths}\label{sec:Mosco}
In this section, we prove the Mosco--convergence of $\boldsymbol{J}_K$ to $\boldsymbol{\mathcal{J}}$ defined in \eqref{eq:c_path} and the convergence of time discrete geodesic paths to a time continuous minimizer of $\boldsymbol{\mathcal{J}}$.
The general procedure follows  
the Mosco--convergence proof in the Euclidean setting \cite{BeEf14}.
Nevertheless we give a comprehensive proof of the convergence result and work out the substantial differences due to the manifold setting. 
These differences are highlighted throughout the proof.
In what follows, we pass to subsequences several times and to increase readability, 
we frequently avoid relabeling subsequences if obvious.
As a first step, we extend the discrete functional~$\boldsymbol{J}_K\colon L^2(\domain,\Hadamard)^{K+1}\times(\mathcal{A}_\varepsilon)^K \to [0, \infty]$ to a functional $\boldsymbol{\mathcal{J}}_K \colon L^2([0,1],L^2(\domain,\Hadamard))\to[0,\infty]$ by
\begin{equation}
\boldsymbol{\mathcal{J}}_K(\image) =
\begin{cases}
\displaystyle\inf_{\boldsymbol{\overline{\varphi}}_K\in(\mathcal{A}_\varepsilon)^K}\left\{\boldsymbol{J}_K(\boldsymbol{\image}_K,\boldsymbol{\overline{\varphi}}_K):
\image^{\mathrm{ext}}_K(\boldsymbol{\image}_K, \boldsymbol{\overline{\varphi}}_K)=\image\right\},
&\hspace{-1.5ex}\text{if there exist }(\boldsymbol{\image}_K, \boldsymbol{\varphi}_K)
\text{ such}\\[-1.5ex]
&\hspace{-1.5ex}\text{that }\image=\image^{\mathrm{ext}}_K(\boldsymbol{\image}_K,\boldsymbol{\varphi}_K)\,,
\\[.5em]
+\infty\,,& \hspace{-1.5ex} \text{else}\,.
\end{cases}	
\label{eq:FuncGamma}
\end{equation}
The condition $\image^{\mathrm{ext}}_K(\boldsymbol{\image}_K, \boldsymbol{\varphi}_K) = \image$ has to hold pointwisely for every $t\in [0,1]$
since the involved expressions are continuous in time.
In fact, it is finite only if for the image path~$\image$ a discrete image path
$\boldsymbol{\image}_K\in L^2(\domain,\Hadamard)^{K+1}$ and a vector of deformations $\boldsymbol\varphi_K\in(\mathcal{A}_\varepsilon)^K$
exist such that $\image=\image^{\mathrm{ext}}_K(\boldsymbol{\image}_K,\boldsymbol{\varphi}_K)$.
In this case, the extended energy coincides with the infimum with respect to the deformation vector for fixed~$\boldsymbol{\image}_K$.
The following lemma guarantees that the infimum is actually attained.
\begin{lemma}\label{lemm:FuncGammaAttainment}
	If for the given fixed image path~$\image$ a discrete image path
	$\boldsymbol{\image}_K\in L^2(\domain,\Hadamard)^{K+1}$ and a vector of deformations $\boldsymbol\varphi_K\in(\mathcal{A}_\varepsilon)^K$
	exist such that $\image=\image^{\mathrm{ext}}_K(\boldsymbol{\image}_K,\boldsymbol{\varphi}_K)$, then 
	the infimum with respect to the vector of deformations  in \eqref{eq:FuncGamma} is attained for some $\boldsymbol{\varphi}_K\in(\mathcal{A}_\varepsilon)^K$.
\end{lemma}
\begin{proof}
	Let $\{\boldsymbol{\varphi}_K^j\}_{j\in\N}\subset(\mathcal{A}_\varepsilon)^K$ be a minimizing sequence for $\boldsymbol{\varphi}_K\mapsto\boldsymbol{J}_K(\boldsymbol{\image}_K,\boldsymbol{\varphi}_K)$,
	which satisfies the equality constraint $\image^{\mathrm{ext}}_K(\boldsymbol{\image}_K,\boldsymbol{\varphi}_K^j)=\image$ for every $j\in\N$.
	Due to the reflexivity of $H^m(\domain,\domain)^K$ a subsequence (not relabeled) exists such that $\boldsymbol{\varphi}_K^j\rightharpoonup\boldsymbol{\overline{\varphi}}_K$ in $H^m(\domain,\domain)^K$.
	The weak lower semi-continuity and the coercivity of $\boldsymbol{\varphi}_K \mapsto\boldsymbol{J}_K(\boldsymbol{\image}_K,\boldsymbol{\varphi}_K)$ are shown in \cite[Theorem~4]{NPS17}.
	Hence, it remains to prove the weak closedness of the equality constraint $\image=\image^{\mathrm{ext}}_K(\boldsymbol{\image}_K,\boldsymbol{\varphi}_K)$.
	Since $H^m(\domain,\domain)^K\hookrightarrow C^{1,\alpha}(\overline\domain,\overline\domain)^K$, we can infer the strong convergence of
	$\boldsymbol{\varphi}_K^j\rightarrow\boldsymbol{\varphi}_K$ in $C^{1,\alpha}(\overline\domain,\overline\domain)^K$.
	Using \cref{cor:stet_norm_2}
	we conclude that for every $t\in[0,1]$ and a.e.~$x\in\domain$
	\[
	\image^{\mathrm{ext}}_K(\boldsymbol{\image}_K,\boldsymbol{\overline{\varphi}}_K)(t,x)=\lim_{j\to\infty}\image^{\mathrm{ext}}_K(\boldsymbol{\image}_K,\boldsymbol{\varphi}_K^j)(t,x)=\image(t,x)
	\]
	holds true.
\end{proof}
In what follows, we will always use the symbol $\boldsymbol{\overline{\varphi}}_K$ for the minimizing set of deformations for given $\boldsymbol{\image}_K$.
The ingredients for the Mosco--convergence introduced in~\Cref{def:MoscoConv} are the \ref{eq:Mosco1} (\Cref{thm:liminf}) and the \ref{eq:Mosco2} (\Cref{thm:limsup}). 
\begin{theorem}[liminf-inequality]\label{thm:liminf}
	Under the assumptions \ref{W1}, \ref{W2} and \ref{W3}
	the time discrete path energy $\boldsymbol{\mathcal{J}}_K$ satisfies the \ref{eq:Mosco1}
	for $\boldsymbol{\mathcal{J}}$ with respect to the $L^2([0,1],L^2(\domain,\Hadamard))$-topology.
\end{theorem}
\begin{proof}
	First, let us give a brief outline of the structure of this proof to facilitate reading. Indeed, the different steps of the proof are as follows:
	\begin{enumerate} \setlength{\itemindent}{.2in}
		\item\textit{Identification of the image and deformation families.}
		In the first step, we retrieve $\boldsymbol{\image}_K$ and $\overline{\boldsymbol\varphi}_K$ from the path~$\image_K$.
		\item\textit{Lower semi-continuity of the weak material derivative.} 
		The convergence of the discrete material derivative~$z_K$ to a limit weak material derivative~$z$ is shown and the lower semi-continuity 
		\[
		\int_0^1\int_\domain z^2\dx x\dx t\leq\liminf_{K\to\infty}K\sum_{k=1}^K\int_{\domain}d\left(\image_{K,k-1}(x),\image_{K,k}\circ\overline{\varphi}_{K,k}(x)\right)^2\dx x
		\]
		is verified.
		\item\textit{Lower semi-continuity of the viscous dissipation.}
		The uniform boundedness of the velocity field~$\velocityConstant_K=\velocityConstant_K(\overline{\boldsymbol\varphi}_K)$ in~$K$ is proven,
		which readily implies $\velocityConstant_K\rightharpoonup\velocityPullback$ in $L^2((0,1),\mathcal V)$.
		Then, the relation 
		\[
		\int_0^1\int_\domain L[\velocity,\velocity]\leq\liminf_{K\to\infty}K\sum_{k=1}^K\int_{\domain}\energyDensity(D\overline{\varphi}_{K,k})+\gamma\Vert D^m \overline{\varphi}_{K,k}\Vert^2\dx x
		\]
		is shown.
		\item \textit{Verification of the admissibility of the limit.}
		In the final step, we prove that $(\image,\velocityPullback,Y,z)$ is a solution of \eqref{eq:ODESys1} and \eqref{eq:ODESys2},
		where $Y$ is the flow associated with~$v$ and $\image$ is the limit image path.
	\end{enumerate}
	
	\bigskip
	
	\paragraph{1. Identification of the image and deformation vectors} 
	Let $\image_K\in L^2([0,1],L^2(\domain,\Hadamard))$
	be a sequence which weakly converges to an image path~$\image\in L^2([0,1],L^2(\domain,\Hadamard))$.
	If we exclude the trivial case $\liminf_{K \to \infty} \boldsymbol{\mathcal{J}}_K(\image_K) = \infty$ and eventually pass to a subsequence (without relabeling), we may assume
	\[
	\boldsymbol{\mathcal{J}}_K(\image_K)\leq\overline{\mathcal{J}}<\infty
	\]
	for all $K\in \N$. By definition of $\boldsymbol{\mathcal{J}}_K$ this directly implies $\image_K = \image^{\mathrm{ext}}_K(\boldsymbol{\image}_K,\overline{\boldsymbol\varphi}_K)$
	with $\boldsymbol{\image}_K=(\image_{K,0},\ldots,\image_{K,K})\in L^2(\domain,\Hadamard)^{K+1}$
	and $\overline{\boldsymbol\varphi}_K=(\overline{\varphi}_{K,1},\ldots,\overline{\varphi}_{K,K})\in(\mathcal{A}_\varepsilon)^K$, which is the minimizing deformation in \eqref{eq:FuncGamma} (existence is proven in~\Cref{lemm:FuncGammaAttainment}).
	In particular, by incorporating \cref{lemm:growthControl} we deduce
	\begin{equation}\label{eq:growthControl}
	\max_{k=1,\dots,K}\Vert\overline{\varphi}_{K,k}-\Id\Vert_{C^{1,\alpha}(\overline\domain)}
	\leq C\max_{k=1,\dots,K}\Vert\overline{\varphi}_{K,k}-\Id\Vert_{H^m(\domain)}
	\leq C\theta(\overline{\mathcal{J}} K^{-1})\leq CK^{-\frac{1}{2}}\,.
	\end{equation}
	We denote by $Y_K$, $X_K$, $v_K$ and $z_K$ the discrete quantities associated with $\overline{\boldsymbol\varphi}_K$ defined in~\cref{sec:Extension},
	which exist for~$K$ sufficiently large.
	
	\medskip
	
	\paragraph{2. Lower semi-continuity of the weak material derivative} 
	Let us remark that this step resembles the first step of the proof in the Euclidean setting
	replacing the squared $L^2$-norm by the squared distance in the Hadamard manifold.
	
	A straightforward computation shows
	\begin{align}
	\int_0^1 \int_\domain z_K^2\dx x\dx t
	= &\sum_{k=1}^K \int_{t_{K,k-1}}^{t_{K,k}} \int_{\domain} K^2 d\big(\image_{K,k-1}(x_{K,k}(t,x)),\image_{K,k}\circ \overline{\varphi}_{K,k} (x_{K,k}(t,x))\big)^2 \dx x \dx t\notag\\
	= &\sum_{k=1}^K \int_{t_{K,k-1}}^{t_{K,k}} \int_{\domain} K^2 d\big(\image_{K,k-1}(x),\image_{K,k}\circ \overline{\varphi}_{K,k} (x)\big)^2 \det(Dy_{K,k}(t,x)) \dx x \dx t\,.
	\label{eq:sourcerelation}
	\end{align}
	Next, we want to bound the difference of $\det(Dy_{K,k})$ and $1$ in the $L^\infty$-norm. Thus, we have
	\[
	Dy_{K,k}(t,x)=\onesymbol+K(t-t_{K,k-1})(D\overline{\varphi}_{K,k}(x)-\onesymbol)\,.
	\]
	Then, the Lipschitz continuity of the determinant on the ball~$B_r(\onesymbol)$ with associated radius $r= \sup_K \max_{k=1,\ldots,K}\Vert\overline{\varphi}_{K,k}-\Id\Vert_{H^{m}(\domain)} < \infty$ implies
	\[
	\Vert \det(Dy_{K,k}(t,x)) - 1 \Vert_{L^\infty([t_{K,k-1},t_{K,k})\times \domain)} \leq C \Vert  \overline{\varphi}_{K,k} - \Id \Vert_{C^{1,\alpha}(\overline{\domain})}\,.
	\]
	Hence, we can deduce from \eqref{eq:growthControl} and $t_{K,k} - t_{K,k-1} = K^{-1}$ that
	\begin{align*}
	&\left\vert\sum_{k=1}^KK^2\int_{t_{K,k-1}}^{t_{K,k}}\int_{\domain}d\bigl(\image_{K,k-1}(x),\image_{K,k}\circ\overline{\varphi}_{K,k}(x)\bigr)^2(\det(Dy_{K,k}(t,x))-1)\dx x\dx t\right\vert\\
	\leq&\delta\overline{\mathcal{J}}C\max_{k=1,\dots,K}\Vert\overline{\varphi}_{K,k}-\Id\Vert_{C^{1,\alpha}(\overline{\domain})}
	\leq\delta\overline{\mathcal{J}}CK^{-\frac{1}{2}}\,.
	\end{align*}
	Taking into account the definition of $z_K$ in \eqref{eq:materialDerivativeDef} this ultimately leads to
	\[
	\lim_{K \to \infty} \int_0^1 \int_\domain z_K^2\dx x\dx t
	= \lim_{K \to \infty} K\sum_{k=1}^K \int_{\domain} d\bigl(\image_{K,k-1}(x),\image_{K,k}\circ \overline{\varphi}_{K,k}(x)\bigr)^2 \dx x\,.
	\]
	This also shows the uniform boundedness of $z_K \in L^2((0,1), L^2(\domain))$, which implies the existence of a weakly convergent subsequence
	with limit $z \in L^2((0,1), L^2(\domain))$. Hence, using the weak lower semi-continuity of the norm we get
	\begin{align*}
	\int_0^1 \int_\domain z^2\dx x\dx t &\leq \liminf_{K \to \infty} \int_0^1\int_\domain z_K^2\dx x\dx t\\ 
	&=\liminf_{K \to \infty} K\sum_{k=1}^K \int_{\domain} d\bigl(\image_{K,k-1}(x),\image_{K,k}\circ \overline{\varphi}_{K,k}(x)\bigr)^2 \dx x\,.
	\end{align*}
	
	\medskip
	
	\paragraph{3. Lower semi-continuity of the viscous dissipation}	
	We highlight that this step differs from the corresponding step appearing in \cite{BeEf14} due to the modification
	of the assumption~\ref{W2}, where the overall structure persists.
	
	Note that the velocity fields~$\velocityPullback_K=\velocityPullback_K(\boldsymbol \varphi_K)$
	are not necessarily in $L^2((0,1),\mathcal V)$.
	The sequence $\velocityConstant_K=\velocityConstant_K(\boldsymbol \varphi_K) \in L^2((0,1),\mathcal V)$ is uniformly bounded in $L^2((0,1),\mathcal V)$. To see this, we first assume that $K$ is sufficiently large
	such that $\max_{k=1,\ldots,K}\|D\overline{\varphi}_{K,k}-\onesymbol\|_{C^0(\overline\domain)}<r_\energyDensity$ (see \ref{W2}), which is possible due to \eqref{eq:growthControl}.
	Then, using Korn's inequality, the Poincar\'e inequality as well as \ref{W2}, we obtain
	\begin{align*}
	\int_0^1\int_\domain \Vert \velocityConstant_K \Vert^2\dx x\dx t
	&\leq C\sum_{k=1}^K\int_{t_{K,k-1}}^{t_{K,k}}\int_\domain K^2\Vert (D\overline{\varphi}_{K,k})^\mathrm{sym}-\onesymbol\Vert^2\dx x\dx t\\
	&\leq CK\sum_{k=1}^K\int_\domain\frac{\energyDensity(D\overline{\varphi}_{K,k})}{C_{\energyDensity,1}}\dx x
	\leq\frac{C\overline{\mathcal{J}}}{C_{\energyDensity,1}}\,,\\
	\int_0^1\int_\domain\Vert D^m\velocityConstant_K\Vert^2\dx x\dx t
	&=\sum_{k=1}^K\int_{t_{K,k-1}}^{t_{K,k}}\int_\domain K^2 \Vert D^m(\overline{\varphi}_{K,k}-\Id)\Vert^2\dx x\dx t\\
	&=\sum_{k=1}^K K\int_\domain \Vert D^m\overline{\varphi}_{K,k}\Vert^2\dx x
	\leq\frac{\overline{\mathcal{J}}}{\gamma}\,.
	\end{align*}
	The Gagliardo--Nirenberg inequality implies uniform boundedness of 
	the sequence~$\velocityConstant_K$ in $L^2((0,1),\mathcal V)$.
	By passing to a subsequence (again labeled in the same way) we can deduce
	$\velocityConstant_K\rightharpoonup \velocity\in L^2((0,1),\mathcal{V})$ for $K\rightarrow\infty$. 
	
	It remains to verify the lower semi-continuity of the sum of the approximate Riemannian distances in~\eqref{eq:d_path}, i.e.
	\begin{equation}\label{eq:liminfL}
	\int_0^1 \int_\domain L[\velocity,\velocity] \leq \liminf_{K \to \infty} K \sum_{k=1}^K \int_{\domain} \energyDensity(D \overline{\varphi}_{K,k}) + \gamma \Vert D^m \overline{\varphi}_{K,k}\Vert^2 \dx x\,.
	\end{equation}
	The second order Taylor expansion around $t_{K,k-1}$ of the function
	$t\mapsto \energyDensity(\onesymbol+(t-t_{K,k-1})D\velocityConstant_{K,k})$	evaluated at $t=t_{K,k}$ yields
	\begin{align}
	\energyDensity(D\overline{\varphi}_{K,k})=&\energyDensity(\onesymbol)+ K^{-1}D\energyDensity(\onesymbol)(D\velocityConstant_{K,k})
	+\frac{1}{2K^2}D^{2}\energyDensity(\onesymbol)(D\velocityConstant_{K,k},D\velocityConstant_{K,k})+r_{K,k}\notag\\
	=&K^{-2}\left(\frac{\lambda}{2}\left(\tr(\varepsilon[\velocityConstant_{K,k}])\right)^2+\mu\tr(\varepsilon[\velocityConstant_{K,k}]^2)\right)+r_{K,k}\,,
	\label{eq:TaylorEnergy}
	\end{align}
	where $r_{K,k}$ denotes a remainder.
	Here, the lower order terms vanish due to \eqref{eq:energy1} and the last equality follows from \eqref{eq:energy2}.
	The remainder satisfies $\Vert r_{K,k}\Vert\leq CK^{-3}\Vert D\velocityConstant_{K,k}\Vert^3$, which follows
	from Taylor's theorem, the definition of~$\velocityConstant_{K,k}=K(\overline{\varphi}_{K,k}-\Id)$ and the growth estimate given in~\cref{eq:growthControl}.
	Then,
	\begin{align*}
	&K\sum_{k=1}^K\int_\domain\energyDensity(D\overline{\varphi}_{K,k})+\gamma \Vert D^m \overline{\varphi}_{K,k}\Vert ^2\dx x \\
	=& K^{-1}\sum_{k=1}^K\int_\domain\frac{\lambda}{2}(\tr(\varepsilon[\velocityConstant_{K,k}]))^2+\mu\tr(\varepsilon[\velocityConstant_{K,k}]^2)+\gamma\Vert D^m\velocityConstant_{K,k}\Vert ^2\dx x + K\sum_{k=1}^K\int_\domain r_{K,k}\dx x\,,
	\end{align*}
	and the remainder is of order $K^{-\frac{1}{2}}$.
	To see this, we apply \eqref{eq:growthControl}, \cref{lemm:growthControl} and the uniform bound on the energy to deduce
	\begin{align*}
	&\quad K\sum_{k=1}^K\int_\domain\Vert r_{K,k}\Vert\dx x\leq CK\sum_{k=1}^K\int_\domain K^{-3}\Vert D\velocityConstant_{K,k}\Vert ^3\dx x\\
	&\leq CK\max_{k=1,\dots,K}\Vert \overline{\varphi}_{K,k}-\Id\Vert _{C^1(\overline\domain)}
	\sum_{k=1}^K\Vert \overline{\varphi}_{K,k}-\Id\Vert _{H^m(\domain)}^2\\
	&\leq CK\theta(\overline{\mathcal{J}} K^{-1})\sum_{k=1}^K\theta\bigl(\boldsymbol{R}(\image_{K,k-1},\image_{K,k},\overline{\varphi}_{K,k})\bigr)^2\\
	&\leq CK^\frac{1}{2}\sum_{k=1}^K\boldsymbol{R}(\image_{K,k-1},\image_{K,k},\overline{\varphi}_{K,k})
	\leq C\overline{\mathcal{J}} K^{-\frac{1}{2}}\,.
	\end{align*}
	Finally, a standard weak lower semi-continuity argument \cite[Theorem~3.20]{Da08} shows 
	\begin{align*}
	&\quad\liminf_{K\to \infty}K\sum_{k=1}^K\int_\domain\energyDensity(D\overline{\varphi}_{K,k})+\gamma\Vert D^m\overline{\varphi}_{K,k}\Vert ^2\dx x\\
	&=\liminf_{K\to \infty} \int_0^1 \int_\domain \frac{\lambda}{2}(\tr \varepsilon[\velocityConstant_K])^2+\mu\tr(\varepsilon[\velocityConstant_K]^2)+\gamma\Vert D^m \velocityConstant_K\Vert ^2 \dx x \dx t\\
	&\geq\int_0^1 \int_\domain \frac{\lambda}{2}(\tr \varepsilon[\velocity])^2+\mu\tr(\varepsilon[\velocity]^2)+\gamma\Vert D^m \velocity\Vert ^2 \dx x \dx t\,,
	\end{align*}
	which implies weak lower semi-continuity
	of the path energy for the sequence~$\{\image_K\}_{K\in\N}$.
	
	\medskip
	
	\paragraph{4. Verification of the admissibility of the limit}
	Finally, it remains to verify that $(\image,\velocityPullback,Y,z)$ for a suitable $Y$ is a solution of \eqref{eq:ODESys1} and \eqref{eq:ODESys2}.
	We have already pointed out that the manifold-valued metamorphosis energy functional necessitates a variational inequality,
	which results in significant modifications of this step compared to \cite{BeEf14}.

	Let $\tilde Y$ denote the solution of
	\begin{equation}
	\begin{array}{rll}
	\displaystyle{\frac{\dx}{\dx t}}\tilde Y(t,x)& =\velocity(t,\tilde Y(t,x)) &\text{for }(t,x) \in [0,1]\times\domain\,,\\[.5em]
	\tilde Y(0,x)&= x  &\text{for } x\in \domain\,,
	\end{array}\label{eq:FlowProof}
	\end{equation}
	which exists due to \cref{thm:DiffeoVelo}. Furthermore, \eqref{eq:velocityEstimate} and the uniform boundedness of $\velocityConstant_K \in L^2((0,1),\mathcal V)$
	imply that the sequence $\velocityPullback_K$ is uniformly bounded in $L^2((0,1),C^{1,\alpha}(\overline{\domain}))$.
	Incorporating \cref{rem:DiffeoVeloRem} we infer that $Y_K$ is uniformly bounded in $C^{0}([0,1],C^{1,\alpha}(\overline{\domain}))$,
	and by exploiting H\"older's inequality we can even show that the sequence is uniformly bounded in $C^{0,\frac{1}{2}}([0,1],C^{1,\alpha}(\overline{\domain}))$.
	Hence, by using the compact embedding of H\"older spaces, the sequence $Y_K$ converges strongly to some $Y$ in $C^{0,\beta}([0,1],C^{1,\beta}(\overline{\domain}))$ for $\beta=\frac{1}{2}\min(\frac{1}{2},\alpha)$.
	
	It remains to verify that $\tilde{Y} = Y$.
	To this end, the solutions of \eqref{eq:FlowProof} corresponding to $\velocityConstant_K$ are denoted by $\tilde{Y}_K$.
	Then,
	\[
	\Vert Y-\tilde Y\Vert_{C^0([0,1]\times\overline{\domain})}\leq\Vert Y-Y_K\Vert_{C^0([0,1]\times \overline{\domain})}+\Vert Y_K-\tilde{Y}_K\Vert_{C^0([0,1]\times\overline{\domain})}+\Vert\tilde{Y}_K-\tilde Y\Vert_{C^0([0,1]\times\overline{\domain})}\,.
	\]
	Here, the first term converges to zero as shown above and the last term converges to zero by the continuous dependence of $\tilde{Y}_K$ on $w_K$ discussed in \cref{thm:DiffeoVelo}.
	Then, we can estimate as follows
	\begin{align}
	\Vert Y_K-\tilde{Y}_K\Vert_{C^0([0,1]\times \overline{\domain})} &\leq C\sum_{k=1}^K\int_{t_{K,k-1}}^{t_{K,k}}\Vert\velocityConstant_{K,k}(s,x_{K,k}(s,\cdot))-\velocityConstant_{K,k}(s,\cdot)\Vert_{C^{0}(\overline{\domain})}\dx s\label{eq:YK1}\\
	&\leq C\sum_{k=1}^K \int_{t_{K,k-1}}^{t_{K,k}}\Vert\velocityConstant_{K,k}(s,\cdot)\Vert_{H^m(\domain)}\Vert y_{K,k}(s,\cdot)-\Id\Vert_{C^{0}(\overline{\domain})}\dx s\notag\\
	&\leq C\Vert\velocityConstant_K\Vert_{L^2((0,1),H^m(\domain))}\max_{k=1,\dots,K}\Vert\overline{\varphi}_{K,k}-\Id \Vert_{C^{0}(\overline{\domain})}\,.\notag
	\end{align}
	Here, the first inequality is deduced from \cref{rem:DiffeoVeloRem}.
	Furthermore, to derive the second inequality we exploit the Lipschitz property of
	$x\mapsto\velocityConstant_{K,k}(s,x_{K,k}(s,x))-\velocityConstant_{K,k}(s,x)$, where the Lipschitz constant is bounded by $C\Vert\velocityConstant_{K,k}(s,\cdot)\Vert_{H^m(\domain)}$,
	and apply the coordinate transform $y_{K,k}(s,\cdot)$.
	The uniform control of $w_K$ and  \eqref{eq:growthControl}  imply $Y = \tilde{Y}$ and by H\"older's inequality
	$Y \in C^{0,\frac{1}{2}}([0,1],C^{1,\alpha}(\overline{\domain}))\,$. Finally, $X_K$ is uniformly bounded in 
	$C^{0,\frac{1}{2}}([0,1],C^{1,\alpha}(\overline{\domain}))$ due to \cref{rem:DiffeoVeloRem}. 
	Thus, \eqref{eq:ODESys1} is fulfilled. 
	
	Next, note that for $s,t\in[0,1]$ we obtain
	\begin{align*}
	\int_{\domain}d\big(\image_K(t,Y_K(t,x)),\image_K(s,Y_K(s,x))\big)^2\dx x 
	&\leq\int_{\domain}\left(\int_t^sz_K(r,Y_K(r,x))\dx r\right)^2\dx x\\
	&\leq \vert s-t\vert\left|\int_{\domain}\int_t^sz_K(r,Y_K(r,x))^2\dx r\dx x\right|\,.
	\end{align*}
	By the uniform boundedness
	of $z_K$ in $L^2((0,1),L^2(\domain))$ we achieve that $\image_K\circ Y_K\in A_{\frac{1}{2},L,\vert\det DY\vert}$ for some appropriate~$L$.
	Next, we verify the weak convergence of a subsequence of $\image_K\circ Y_K$ to $\image\circ Y\in A_{\frac{1}{2},L,\vert \det DY \vert}$.
	To this end, we observe 
	\begin{align*}
	\limsup_{K \to \infty}\dist_2(\image_K,\image)^2 &=\limsup_{K \to \infty} \int_0^1 \int_{\domain} d\big(\image_K(t,Y_K(t,x)),\image(t,Y_K(t,x))\big)^2 \vert \det DY_K \vert \dx x \dx t\\
	&= \limsup_{K \to \infty} \int_0^1 \int_{\domain} d\big(\image_K(t,Y_K(t,x)),\image(t,Y(t,x))\big)^2 \vert \det DY \vert \dx x \dx t\,.
	\end{align*}
	For the first equality we incorporate the transformation formula,
	the second equality follows from the uniform convergence of $DY_K$, the metric triangle inequality and the 
	convergence of $\image(t,Y_K(t,x))$ to $\image(t,Y(t,x))$ (see \cref{lemm:stet_norm}).
	To sum up, this proves the weak convergence of $\image_K\circ Y_K$ according to \cref{eq:WeakConv} and 
	by \cref{thm:LipClos}, the limit is also contained in $A_{\frac{1}{2},L,\vert \det DY \vert}$.
	
	Finally, it remains to verify \eqref{eq:ODESys2}. 
	Assume there exist $s<t\in[0,1]$ such that the set
	\[
	B \coloneqq\left\{x\in\domain\colon d\big(\image(s,Y(s,x)),\image(t,Y(t,x))\big)>\int_s^t z(r,Y(r,x)) \dx r\right\}
	\]
	has positive Lebesgue measure. 
	From the joint convexity of the metric $d(\cdot,\cdot)$ and the continuity of point evaluation in time, we get that the functional
	$\image \mapsto\int_B d(\image(s,x),\image(t,x))\dx x$ 
	is continuous and convex on $A_{\frac{1}{2},L,\vert \det DY \vert}$.
	Now, this  
	implies weak lower semi-continuity of the mapping, see \cite[Lemma 3.2.3]{Bac14}, and we obtain
	\begin{align*}
	&\int_B d\big(\image(s,Y(s,x)),\image(t,Y(t,x))\big)\dx x\leq\liminf_{K \to \infty} \int_B d\big(\image_K(s,Y_K(s,x)),\image_K(t,Y_K(t,x))\big) \dx x\\
	\leq &\liminf_{K \to \infty} \int_B \int_s^t z_K(r,Y_K(r,x)) \dx r\dx x =\int_B \int_s^t z(r,Y(r,x)) \dx r\dx x\,,
	\end{align*}
	where the last equality follows from the weak convergence of $z_K$ combined with the strong convergence of $Y_K$, which also implies the weak convergence of $z_K \circ Y_K$.
	This yields a contradiction and concludes the proof of the \ref{eq:Mosco1}.
\end{proof}
In what follows, we prove the existence of a recovery sequence and thus establish the Mosco--convergence.
As a preparation, we prove that the infimum in \eqref{eq:c_path} is actually attained, where we exploit some results of the proof of \Cref{thm:liminf}.
\begin{proposition}\label{prop:ExistTupel}
	For $\image\in L^2([0,1],L^2(\domain,\Hadamard))$ with $\boldsymbol{\mathcal{J}}(\image)<\infty$ the infimum in \eqref{eq:c_path} is attained, i.e.~there exists a tuple $(v,z)\in\sol(\image)$ satisfying \eqref{eq:ODESys1} and \eqref{eq:ODESys2}.
\end{proposition}
\begin{proof}
	We first observe that the functional $(v,z)\mapsto\int_0^1\int_\domain L[v,v]+\regparam z^2\dx x\dx t$ is weakly lower semi-continuous and coercive on $\sol(\image)$, cf.~\cite{BeKr17}.
	Since $\sol(\image)$ is a subset of a reflexive Banach space, it suffices to prove the weak closedness of $\sol(\image)$ to obtain the existence of
	an optimal tuple $(v,z)\in\sol(\image)$.
	
	Let $\{(\velocityPullback_k,z_k)\}_{k\in\N}\in\sol(\image)$ be a weakly convergent sequence with limit $(\velocityPullback,z)$.
	Due to \cref{thm:DiffeoVelo} the corresponding flows~$Y_k$ and $Y$ given by \eqref{eq:ODESys1} exist and $Y_k\to Y$ in $C^0([0,1]\times\overline{\domain})$ and
	the weak convergence of $\velocityPullback_k$ implies the uniform boundedness of $\{v_k\}_{k \in \N}$ in $L^2((0,1),C^{1,\alpha}(\overline{\domain}))$.
	Thus, the reasoning in the paragraph following \eqref{eq:FlowProof} implies that a subsequence of $\{Y_k\}_{k\in\N}$ converges strongly to~$Y$ in~$C^{0,\beta}([0,1],C^{1,\beta}(\overline{\domain}))$ for $\beta=\frac{1}{2}\min(\frac{1}{2},\alpha)$.
	
	Finally, \cref{lemm:stet_norm} implies $\image(t,Y_k(t,x)) \to \image(t,Y(t,x))$ in $L^2([0,1],L^2(\domain,\Hadamard))$ and the last part of the proof of \cref{thm:liminf} 
	shows that $(\image,Y,z)$ is a solution of \eqref{eq:ODESys2}.
\end{proof}
\begin{theorem}[Recovery sequence]\label{thm:limsup}
	Let $\image_A,\image_B\in L^2(\domain,\Hadamard)$ be fixed input images
	and let $\image\in L^2([0,1],L^2(\domain,\Hadamard))$ be an image path with $\image(0)=\image_A$ and $\image(1)=\image_B$.
	Then there exists a recovery sequence~$\{\image_K\}_{K\in\N}$ with $\image_K(0)=\image_A$ and $\image_K(1)=\image_B$ for all $K\in\N$
	such that the \ref{eq:Mosco2} in \Cref{def:MoscoConv} w.r.t.~the $L^2([0,1],L^2(\domain,\Hadamard))$-topology is valid.
\end{theorem}
\begin{proof} We proceed in three steps, which follow the usual general guideline to show 
	existence of recovery sequences in the context of $\Gamma$ convergence:\medskip
	\begin{enumerate}\setlength{\itemindent}{.2in}
		\item \textit{Construction of the recovery sequence}.
		\item \textit{Verification of the limsup-inequality}.
		\item \textit{Identification of the recovery sequence limit}.
	\end{enumerate}
	\smallskip
	
	\paragraph{1. Construction of the recovery sequence}
	Compared to \cite{BeEf14} our construction avoids the approximation of $v$ and defines the deformations directly.
	Due to \cref{prop:ExistTupel}, there exist optimal $(v, Y, z)$ corresponding to $\image$ satisfying \eqref{eq:ODESys1} and \eqref{eq:ODESys2}.
	Incorporating the flow $Y$, we define for given $K\in \N$ a vector of diffeomorphisms
	$\boldsymbol{\varphi}_K=(\varphi_{K,1},\ldots,\varphi_{K,K})\in H^m(\domain,\R^n)^K$ by
	\[
	\varphi_{K,k} = Y_{t_{K,k-1}}(t_{K,k},\cdot)\,,
	\]
	where $Y_{a}(b,\cdot) \coloneqq Y(b,Y^{-1}(a,\cdot)) \in H^m(\domain)^K$ with $a,b\in[0,1]$.
	This expression coincides with the evaluation at $t=1$ of the flow corresponding to the velocity field $\velocity_{a,b}(t,x)\coloneqq(b-a)\velocity(a+(b-a)t,x)$, i.e.~the solution of
	\begin{align}
	\displaystyle{\frac{\dx}{\dx t}}Y_{a,b}(t,x)&=\velocity_{a,b}(t,Y_{a,b}(t,x))&&\text{for }(t,x)\in[0,1]\times\domain\,,\label{eq:FlowRescaled}\\
	Y_{a,b}(0,x)&=x &&\text{for }x\in\domain\,.\notag
	\end{align}
	Here, $v$ is the velocity field whose existence is postulated in \cref{prop:ExistTupel}.
	Next, we bound the $C^1(\overline\domain)$-norm of the displacements as follows:
	\begin{align}
	&\max_{k \in \{1,\dots,K\}}\Vert\varphi_{K,k}-\Id\Vert_{C^1(\overline\domain)}\notag\\
	\leq&\sup_{\substack{s,t\in[0,1]\\ \vert t-s\vert\leq K^{-1}}}\Vert Y_{s}(t,\cdot)-\Id \Vert_{C^1(\overline \domain)}
	\leq\sup_{\substack{s,t\in[0,1]\\ \vert t-s\vert\leq K^{-1}}} C\int_{0}^{1} \Vert \velocity_{s,t}\bigl(r,Y_{s,t}(r,\cdot)\bigr)\Vert_{H^m(\domain)}\dx r\notag\\
	\leq&\sup_{\substack{s,t\in[0,1]\\ \vert t-s\vert\leq K^{-1}}}C\left\vert\int_{s}^{t}\Vert\velocity(r,\cdot)\Vert_{H^m(\domain)}\dx r\right\vert
	\leq CK^{-\frac{1}{2}}\sup_{ \substack{s,t\in[0,1]\\ \vert t-s\vert\leq K^{-1}}}\left\vert\int_s^t\Vert\velocity(r,\cdot)\Vert_{H^m(\domain)}^2\dx r\right\vert^\frac{1}{2}\,.\label{eq:growthControl2}
	\end{align}
	For the third inequality, we exploit the estimate
	\begin{equation}
	\|v(t,Y(t,\cdot))\|_{H^m(\domain)}\leq C\|v(t,\cdot)\|_{H^m(\domain)}\,,
	\label{eq:concatenationEstimateFlow}
	\end{equation}
	which follows from \cite[Lemma~3.5]{BrVi16} and an extension argument as shown in~\cref{thm:DiffeoVelo}, and use the transformation formula.
	The last inequality is implied by the Cauchy--Schwarz inequality.
	
	Choosing $K$ sufficiently large ensures $\boldsymbol{\varphi}_K \in(\mathcal{A}_\varepsilon)^K$ and we can apply the temporal extension from \cref{sec:Extension}.
	Finally, the recovery sequence is defined as
	$\image_K = \image^{\mathrm{ext}}_K\left(\boldsymbol{\image}_K,\boldsymbol{\varphi}_K\right)$, where
	\[
	\boldsymbol{\image}_K=\left(\image_{K,0},\image_{K,1},\dots,\image_{K,K}\right)=\left(I(t_{K,0},\cdot),\dots,I(t_{K,K},\cdot)\right).
	\]
	
	\medskip
	
	\paragraph{2. Verification of the limsup-inequality} 
	Note that this step shares some similarities with the corresponding step in \cite{BeEf14} with modifications necessitated by the manifold structure and the different construction.
	In the following, all terms in the discrete energy $\boldsymbol{J}_K(\boldsymbol{\image}_K,\boldsymbol{\varphi}_K)$ are estimated separately.
	For any $k=1,\ldots,K$ we infer using \eqref{eq:ODESys2}, Jensen's inequality and \eqref{eq:growthControl2} that
	\begin{align}
	&\int_\domain d\big(\image_{K,k-1},\image_{K,k}\circ\varphi_{K,k}\big)^2\dx x\notag\\
	=&\int_\domain d\big(\image_{K,k-1}\circ Y(t_{K,k-1},x),\image_{K,k}\circ Y(t_{K,k},x)\big)^2\det(DY(t_{K,k-1},x))\dx x\notag\\
	\leq&\int_\domain\left(\int_{t_{K,k-1}}^{t_{K,k}}z(s,Y(s,x))\dx s\right)^2\det(DY(t_{K,k-1},x))\dx x\notag\\
	\leq&\frac{1}{K}\int_{t_{K,k-1}}^{t_{K,k}}\int_\domain z^2(s,x)\det(DY_{s}(t_{K,k-1},x))\dx x\dx s\notag\\
	\leq&\frac{1}{K}\left(1+CK^{-\frac{1}{2}}\right)\int_{t_{K,k-1}}^{t_{K,k}}\int_\domain z^2(s,x)\dx x\dx s\,.\label{eq:zBoundIII}
	\end{align}
	Recall that $w_{K,k} = K(\varphi_{K,k}-\Id)$.
	Now, the same Taylor argument as in \cref{eq:TaylorEnergy} implies
	\begin{align}
	&\int_\domain\energyDensity(D\varphi_{K,k})+\gamma\Vert D^m\varphi_{K,k}\Vert ^2 \dx x \leq K^{-2}\int_\domain L[\velocityConstant_{K,k},\velocityConstant_{K,k}]\dx x + CK^{-3}\int_\domain \Vert D\velocityConstant_{K,k}\Vert ^3\dx x\,.
	\label{eq:remainderEllipticOperator}
	\end{align}
	Summing over the second term on the right hand side and taking into account \eqref{eq:growthControl2} we obtain
	\begin{align}
	&\sum_{k=1}^K \int_\domain \Vert D\velocityConstant_{K,k}\Vert^3\dx x
	\leq C K^3 \sum_{k=1}^K \Vert \varphi_{K,k} - \Id\Vert_{C^1(\overline \domain)}^3 \leq CK^\frac{3}{2}\,.
	\label{eq:estimateRemainderGrowth}
	\end{align}
	A direct application of Jensen's inequality shows that the lower order term satisfies
	\begin{align}
	\int_\domain L[\velocityConstant_{K,k},\velocityConstant_{K,k}]\dx x
	=& \int_\domain L\!\left[ K \!\int_{t_{K,k-1}}^{t_{K,k}} \!\!\velocity(t,Y_{t_{K,k-1}}(t,x))\dx t, K \!\int_{t_{K,k-1}}^{t_{K,k}} \!\!\velocity(t,Y_{t_{K,k-1}}(t,x))\dx t\right]\!\!\dx x \notag\\
	\leq& \int_\domain K\!\int_{t_{K,k-1}}^{t_{K,k}} L[\velocity(t,Y_{t_{K,k-1}}(t,x)),\velocity(t,Y_{t_{K,k-1}}(t,x))]\dx t\dx x\,.
	\label{eq:JensenEllipticOperator}
	\end{align}
	By using \eqref{eq:growthControl2} and $\vert\tr(AB)\vert\leq\vert\tr(A)\vert+\vert\tr(A(B-\Id))\vert$ for~$A,B\in\R^{n\times n}$
	multiple times we can estimate the part corresponding to the first summand of $L$, see \eqref{eq:ellipticOperator}, as follows:
	\begin{align}
	&\int_\domain\int_{t_{K,k-1}}^{t_{K,k}}\tr\Big(D(\velocity(t,Y_{t_{K,k-1}}(t,x)))\Big)^2\dx t\dx x\notag\\
	=&\int_\domain\int_{t_{K,k-1}}^{t_{K,k}}\tr\Big(D\velocity(t,Y_{t_{K,k-1}}(t,x)) DY_{t_{K,k-1}}(t,x)\Big)^2\dx t\dx x\notag\\
	\leq&\int_\domain\int_{t_{K,k-1}}^{t_{K,k}}\tr\Big(D\velocity(t,Y_{t_{K,k-1}}(t,x))\Big)^2
	+\tr\Big(D\velocity(t,Y_{t_{K,k-1}}(t,x))(\Id-Y_{t_{K,k-1}}(t,x))\Big)^2\notag\\
	&+2\left\vert\tr\Big(D\velocity(t,Y_{t_{K,k-1}}(t,x))\Big)\tr\Big(D\velocity(t,Y_{t_{K,k-1}}(t,x))(\Id-Y_{t_{K,k-1}}(t,x))\Big)\right\vert\dx t\dx x\notag\\
	\leq&\int_\domain\int_{t_{K,k-1}}^{t_{K,k}}\tr\Big(D\velocity\bigl(t,Y_{t_{K,k-1}}(t,x))\Big)^2+C(1+\Vert v(t,\cdot)\Vert_{H^m(\domain)}^3)K^{-\frac{1}{2}}\dx t\dx x\notag\\
	\leq&\int_\domain\int_{t_{K,k-1}}^{t_{K,k}}\tr(\varepsilon[v])^2+C(1+\Vert v(t,\cdot)\Vert_{H^m(\domain)}^3)K^{-\frac{1}{2}}\dx t\dx x\,.
	\end{align}
	For the last inequality, we additionally used the transformation formula and~\eqref{eq:growthControl2}.
	The second term in~$L$ is estimated analogously:
	\begin{align}
	&\int_\domain\int_{t_{K,k-1}}^{t_{K,k}}\tr\Big(\varepsilon(\velocity(t,Y_{t_{K,k-1}}(t,x)))^2\Big)\dx t\dx x\notag\\
	\leq &\int_\domain\int_{t_{K,k-1}}^{t_{K,k}}\tr(\varepsilon[v]^2)+C(1+\Vert v(t,\cdot)\Vert_{H^m(\domain)}^3)K^{-\frac{1}{2}}\dx t\dx x\,.
	\end{align}
	It remains to bound the higher order term appearing in the definition of~$L$.
	To this end, we use \eqref{eq:concatenationEstimateFlow} and the bound
	$\Vert fg\Vert_{H^{\tilde m}}\leq C\Vert f\Vert_{H^m}\Vert g\Vert_{H^{\tilde m}}$
	for $f\in H^m(\domain)$, $g\in H^{\tilde m}(\domain)$ and any $0\leq \tilde m\leq m$, see \cite[Lemma~2.3]{InKaTo13}, which results in the estimates
	\begin{align*}
	&\vert v\big(t,Y_{t_{K,k-1}}(t,\cdot)\big)\vert_{H^m(\domain)}\\ \leq\,&\vert Dv\big(t,Y_{t_{K,k-1}}(t,\cdot)\big) \vert_{H^{m-1}(\domain)} + \big\Vert Dv\big(t,Y_{t_{K,k-1}}(t,\cdot)\big)D(Y_{t_{K,k-1}}(t,\cdot)-\Id)\big\Vert_{H^{m-1}(\domain)}\\
	\leq\,&\vert Dv\big(t,Y_{t_{K,k-1}}(t,\cdot)\big) \vert_{H^{m-1}(\domain)} + C\big\Vert v(t,\cdot) \Vert_{H^{m}(\domain)} \big\Vert Y_{t_{K,k-1}}(t,\cdot) - \Id \Vert_{H^{m}(\domain)}\\
	\leq\,&\vert Dv\big(t,Y_{t_{K,k-1}}(t,\cdot)\big) \vert_{H^{m-1}(\domain)} + C\big\Vert v(t,\cdot) \Vert_{H^{m}(\domain)}K^{-\frac{1}{2}}\,.
	\end{align*}
	By iterating this argument and applying a change of variables we obtain for the last term of~$L$
	\begin{align}
	\int_{t_{K,k-1}}^{t_{K,k}} \big\vert \velocity\big(t,Y_{t_{K,k-1}}(t,\cdot)\big)\big\vert^2_{H^m(\domain)} \dx t \leq \int_{t_{K,k-1}}^{t_{K,k}} \vert v(t,\cdot)\vert^2_{H^m(\domain)} + C\bigl\Vert v(t,\cdot) \Vert^2_{H^{m}(\domain)}K^{-\frac{1}{2}} \dx t\,.\label{eq:LastEstimate}
	\end{align}
	By combining the estimate \eqref{eq:zBoundIII} with \eqref{eq:remainderEllipticOperator}--\eqref{eq:LastEstimate} for the second inequality below we get
	\begin{align*}
	\mathcal{J}_K(\image_K)
	&\leq K\sum_{k=1}^K\int_\domain\energyDensity(D\varphi_{K,k})+\gamma\vert D^m\varphi_{K,k}\vert ^2+\regparam d\big(\image_{K,k-1},\image_{K,k}\circ\varphi_{K,k}\big)^2\dx x\\
	&\leq \sum_{k=1}^K\left(\int_{t_{K,k-1}}^{t_{K,k}}\int_\domain L[\velocityPullback,\velocityPullback]+CK^{-1}\vert D\velocityConstant_{K,k}\vert ^3+
	\regparam\left(1+CK^{-\frac{1}{2}}\right)z^2(t,x)\dx x\dx t\right)\\
	&\leq\int_0^1\int_\domain L[\velocity,\velocity]+\regparam z^2(t,x)\dx x\dx t+CK^{-\frac{1}{2}}+C\regparam K^{-\frac{1}{2}}=\mathcal{J}(\image)+\bigO(K^{-\frac{1}{2}})\,,
	\end{align*}
	which readily implies the \ref{eq:Mosco2}.\medskip
	
	\paragraph{3. Identification of the recovery sequence limit}
	It remains to verify the convergence $\image_K\to \image$ in $L^2([0,1],L^2(\domain,\Hadamard))$ as $K\to\infty$.
	To see this we estimate
	\begin{align*}
	&\int_0^1\int_\domain d\big(\image(s,Y(s,x)),\image^{\mathrm{ext}}_K(\boldsymbol{\image}_K, \boldsymbol{\varphi}_K)(s,Y_K(s,x))\big)^2\dx x \dx s\\
	=&\sum_{k=1}^K \int_{t_{K,k-1}}^{t_{K,k}}\int_\domain d\big(\image(s,Y(s,x)),\image^{\mathrm{ext}}_K(\boldsymbol{\image}_K, \boldsymbol{\varphi}_K)(s,Y_K(s,x))\big)^2\dx x \dx s\\
	\leq & C \sum_{k=1}^K\biggl( \int_{t_{K,k-1}}^{t_{K,k}} \int_\domain K^{-2}z^2(s,Y(s,x)) \dx x \dx s\\
	&\hspace{3em}+ \int_{t_{K,k-1}}^{t_{K,k}}\int_\domain d\big(\image_{K,k-1}(Y(t_{K,k-1},x)),\image^{\mathrm{ext}}_K(\boldsymbol{\image}_K, \boldsymbol{\varphi}_K)(s,Y_K(s,x))\big)^2
	\dx x\dx s\biggr)\\
	\leq &CK^{-2} \Vert z(t,Y(t,x))\Vert_{L^2((0,1) \times \domain)}^2\,.
	\end{align*}
	Here, we combined \eqref{eq:ODESys2} with the Cauchy--Schwarz inequality to obtain an estimate for the term $d(I(s,Y(s,x)),\image_{K,k}(Y(t_{K,k-1},x)))$ in the first inequality and used
	the definition of $\image^{\mathrm{ext}}_K$, see \eqref{eq:DefUk}, together with \eqref{eq:distancePropertyGeodesics}, \eqref{eq:ODESys2} and the Cauchy--Schwarz inequality in the second inequality.
	Due to the convergence of $Y_K$ to $Y$ and \cref{cor:stet_norm_2}, this readily implies the claimed convergence $\image^{\mathrm{ext}}_K(\boldsymbol{\image}_K, \boldsymbol{\varphi}_K) \to \image$.
\end{proof}
We conclude this section with the desired convergence statement for discrete geodesic paths.
\begin{theorem}[Convergence of discrete geodesic paths]\label{thm:convergence}
	Let $\image_A,\image_B\in L^2(\domain,\Hadamard)$ and suppose that
	the assumptions \ref{W1}, \ref{W2} and \ref{W3} hold true.
	For every $K\in\N$ let $\image_K$ be a minimizer of $\boldsymbol{\mathcal{J}}_K $
	subject to $\image_K(0)=\image_A$ and $\image_K(1)=\image_B$.
	Then, a subsequence of $\{\image_K\}_{K\in \N}$ converges weakly in $L^2([0,1],L^2(\domain,\Hadamard))$
	to a minimizer of the continuous path energy $\boldsymbol{\mathcal{J}}$ as $K\rightarrow\infty$, and the associated sequence of discrete energies converges to the minimal continuous path energy.
\end{theorem}
\begin{proof}
	Using a comparison argument with $v\equiv 0$ and $z(t,x) = \dist_2(\image_A(x),\image_B(x))$ we deduce that the path energy $\boldsymbol{\mathcal{J}}_K$ is
	bounded by $\overline{\mathcal{J}} = \regparam\dist_2(\image_A,\image_B)^2$.
	For optimal vectors of images~$\boldsymbol{\image}_K$ and 
	deformations~$\boldsymbol{\varphi}_K$ in the definition of $\boldsymbol{\mathcal{J}}_K$, see \eqref{eq:FuncGamma},  we apply the temporal extension construction 
	from \cref{sec:Extension}.
	In particular,
	$\boldsymbol{J}_K(\boldsymbol{\image}_K,\boldsymbol{\varphi}_K) \leq \overline{\mathcal{J}}$
	for all $K \in \N$.
	Using \eqref{eq:BoundInverse} and \cref{eq:sourcerelation}, we conclude that $z_K$ is uniformly bounded in $L^2((0,1)\times \domain)$.
	Next, \cref{rem:DiffeoVeloRem} together with \eqref{eq:velocityEstimate} and \cref{eq:growthControl} imply
	the uniform boundedness of $Y_K$, $X_K$ in $C^{0}([0,1],C^{1,\alpha}(\overline{\domain}))$.
	Incorporating \eqref{eq:ODESys2}, we obtain for $f_a(x) = a$ with $a \in \Hadamard$ that
	\[
	\dist_2(\image_K(t,\cdot),f_a)
	\leq C \bigl(\dist_2\big(\image_K(t,Y_K(t,\cdot)),\image_A\big) + \dist_2(\image_A,f_a)\bigr)
	\leq C\big(\Vert z_K \Vert_{L^2((0,1)\times \domain)} + 1 \big)\,.
	\]
	Therefore, $\{\image_K\}_{K\in\N}$ is uniformly bounded in $L^\infty([0,1],L^2(\domain,\Hadamard))$
	and a subsequence converges weakly to some $\image\in L^2([0,1],L^2(\domain,\Hadamard))$ in $L^2([0,1],L^2(\domain,\Hadamard))$.
	
	Now, we follow the usual argument and assume that there exists an image path $\tilde\image\in L^2([0,1],L^2(\domain,\Hadamard))$
	with corresponding optimal tuple $(\tilde\image,\tilde v, \tilde Y,z)$, which exists due to~\cref{prop:ExistTupel}, satisfying \eqref{eq:ODESys1} and \eqref{eq:ODESys2} such that
	\begin{equation}
	\boldsymbol{\mathcal{J}}[\tilde\image]<\boldsymbol{\mathcal{J}}[\image]\,,
	\label{eq:gammaContradiction}
	\end{equation}
	By \cref{thm:limsup}, we see that there exists a sequence
	$\{\tilde\image_K\}_{K\in\N}\subset L^2((0,1),L^2(\domain,\Hadamard))$ satisfying $\limsup_{K\rightarrow\infty}\boldsymbol{\mathcal{J}}_K [\tilde\image_K]\leq\boldsymbol{\mathcal{J}}[\tilde\image]$.
	Thus, we obtain applying \cref{thm:liminf}
	\begin{equation}\label{eq:limit}
	\boldsymbol{\mathcal{J}}[\image]\leq\liminf_{K\rightarrow\infty}\boldsymbol{\mathcal{J}}_K [\image_K]
	\leq\limsup_{K\rightarrow\infty}\boldsymbol{\mathcal{J}}_K [\tilde\image_K]\leq\boldsymbol{\mathcal{J}}[\tilde \image]\,,
	\end{equation}
	which contradicts \eqref{eq:gammaContradiction}. Hence, $\image$ minimizes the continuous path energy over all admissible image paths. Finally,
	the discrete path energies converge to the limiting path energy along a subsequence, i.e.~$\lim_{K\rightarrow\infty}\boldsymbol{\mathcal{J}}_K [\image_K]=\boldsymbol{\mathcal{J}}[\image]$,
	which again follows from \cref{eq:limit} by using $\tilde \image = \image$.
\end{proof}

\section{Conclusion}\label{sec:conclusion}
In this paper, we have introduced a novel metamorphosis functional for manifold-valued images. 
We specifically considered the case of images as maps into Hadamard manifolds.
This choice is at first motivated by applications like DTI images, which we depicted as examples here.
On the other hand, Hadamard manifolds come with the joint convexity of the distance functional. 
An important aspect of the generalized metamorphosis model for manifold-valued images is the inequality~\eqref{eq:ODESys2}, 
which replaces the defining equation for the material derivative in the standard metamorphosis model.
As it is shown here, it is in particular the joint convexity of the distance function which allows us to show this inequality as the limit inequality for our discrete approximation. 
Thus, Hadamard manifolds naturally arise in applications and appear to be the proper setup for which the existence and convergence analysis is still possible.
Indeed, we picked up a natural time discretization for this model and proved the Mosco--convergence to this novel time continuous metamorphosis model.
This in particular establishes the existence of solutions for this model, not following or using the approach by Trouv\'e and Younes in~\cite{TrYo05a}.
Also numerically, the joint convexity of the distance on Hadamard manifolds is of importance for the convergence of the alternating descent scheme presented in~\cite{NPS17}.

\section*{Acknowledgments}
We gratefully acknowledge Johannes Persch and Gabriele Steidl for many enlightening discussions and inspirations.
Furthermore, we thank the anonymous referees for making valuable comments and hints to improve the paper.
S.~Neumayer acknowledges funding by the Research Training  Group  1932, project  area  P3.
A.~Effland and M.~Rumpf acknowledge support of the Collaborative Research Center 1060 funded by 
the Deutsche Forschungsgemeinschaft (DFG, German Research Foundation)  
and the Hausdorff Center for Mathematics, funded by the Deutsche Forschungsgemeinschaft (DFG, German Research Foundation) 
under Germany's Excellence Strategy - GZ 2047/1, Projekt-ID 390685813.
A.~Effland additionally acknowledges support from the European Research Council under the Horizon 2020 program, ERC starting grant HOMOVIS, No. 640156.

\bibliography{references}
\bibliographystyle{abbrv}
\end{document}